\documentclass{article}

\usepackage[utf8]{inputenc}
\usepackage[T1]{fontenc}
\usepackage{csquotes}
\usepackage[american]{babel}
\usepackage{amsthm,amsmath,mathtools,mathrsfs,latexsym}
\usepackage{amssymb}
\usepackage{dsfont}
\usepackage{comment}
\usepackage{color}
\usepackage{enumitem}
\usepackage[colorlinks=true,linkcolor=blue,citecolor=blue]{hyperref}

\usepackage[a4paper, top=2cm, bottom=2.5cm, left=2.5cm, right=2.5cm]{geometry}

\usepackage[backend=biber,style=alphabetic,sorting=nyt,natbib=true]{biblatex}
\DeclareNolabel{\nolabel{\regexp{[\p{P}\p{S}\p{C}\p{Z}]+}}}
\bibliography{Biblio}

\newcommand{\mN}{\mathbb{N}}
\newcommand{\R}{\mathbb{R}}
\newcommand{\mP}{\mathbb{P}}
\newcommand{\mE}{\mathbb{E}}
\newcommand{\id}{\mathbb{I}}

\newcommand{\Ran}{\operatorname{Ran}}
\newcommand{\mi}{\mathds{1}}
\newcommand{\llangle}{{\langle\!\langle}}
\newcommand{\rrangle}{{\rangle\!\rangle}}
\newcommand{\ds}{\displaystyle}
\newcommand{\itemref}[1]{$\textit{\ref{#1}}$}

\renewcommand{\d}{\,\mathrm{d}}
\renewcommand{\div}{\operatorname{div}}

\def\a{{\alpha}}
\def\b{{\beta}}

\def\la{{\lambda}}

\def\g{{\gamma}}
\def\eps{{\varepsilon}}

\newtheorem{theorem}{Theorem}[section]
\newtheorem{proposition}[theorem]{Proposition}
\newtheorem{lemma}[theorem]{Lemma}

\theoremstyle{definition}
\newtheorem{definition}[theorem]{Definition} 
\newtheorem{example}[theorem]{Example} 

\theoremstyle{remark}
\newtheorem{remark}[theorem]{Remark}

\title{\vskip -1cm New \textit{a Priori} Estimate for Stochastic 2D Navier-Stokes               Equations with Applications to Invariant Measure
    } 
\author{
    Matteo Ferrari \\
    {\footnotesize Dept.~of Mathematics $``$Felice Casorati$"$, University of Pavia}\\
    {\footnotesize \texttt{matteo.ferrari13@universitadipavia.it}}\\
    {\footnotesize \texttt{https://orcid.org/0009-0000-6585-2210}}}
\date{\today}

\begin{document}

\maketitle

\begin{abstract}
The paper deals with the stochastic two-dimensional Navier-Stokes equation for incompressible fluids, set in a bounded domain with Dirichlet boundary conditions. We consider additive noise in the form $G\d W$, where $W$ is a cylindrical Wiener process and $G$ a bounded linear operator with range dense in the domain of $A^\gamma$, $A$ being the Stokes operator.  
While it is known that existence of invariant measure holds for $\gamma>1/4$, previous results show its uniqueness only for $\gamma > 3/8$. We fill this gap and prove uniqueness and strong mixing property in the range $\gamma \in (1/4, 3/8]$ by adapting the so-called Sobolevski\u{\i}-Kato-Fujita approach to the stochastic N-S equations. This method provides new \textit{a priori} estimates, which entail both better regularity in space for the solution and strong Feller and irreducibility properties for the associated Markov semigroup.
\\[3mm]

\end{abstract}

\section{Introduction}

The Navier-Stokes equations provide a complete characterization of the motion of a viscous Newtonian fluid. For incompressible and homogeneous fluids, the N-S equations take the following form:
\begin{equation}\label{EQ:archeoNS}
    \begin{cases}
    \partial^{}_tu -\nu \Delta u+(Du)u=-\frac{1}{\rho}\nabla p+f\\
    \div u=0
    \end{cases},
\end{equation}
where the kinematic viscosity $\nu>0$ and the density $\rho>0$ are given constants, while $f$ denotes a known external force acting on the system. Here, the unknowns are $p$ and $u$, which represent respectively the pressure scalar field and the flow velocity vector field of the fluid. We set the equation in a non-empty bounded and connected open set $\mathcal D\subset\mathbb{R}^2$ with Lipschitz boundary. Consequently, we require $u, f: [0, +\infty) \times \overline{\mathcal D} \to \mathbb{R}^2$ and $p: [0, +\infty) \times \overline{\mathcal D} \to \mathbb{R}$. Moreover, we choose the units of measurement such that $\nu=1$. We associate equation \eqref{EQ:protoNS} with the Dirichlet boundary condition and an initial datum
\[
u=0 \quad \text{in } [0,+\infty)\times\partial \mathcal D\qquad\qquad u(0,\cdot)=u_0(\cdot) \quad \text{in } \overline{\mathcal D}.
\]

In order to study equation \eqref{EQ:archeoNS}, we introduce a basic Hilbert space $H\subset L^2(\mathcal D;\mathbb{R}^2)$ (see Section \ref{SEC:functionalsetting}), which incorporates both the boundary condition and the divergence-free condition. In Section \ref{SEC:operators}, we define the Leray projector $\Pi$ onto $H$, the Stokes operator $A=-\Pi\Delta$, and the Navier-Stokes non-linearity $B$ (\textit{cf.} \cite{temam2001navier,vishik1988mathematical,temam1995navier} and the references therein).
In the literature the case when the external force $f$ is random is extensively investigated starting from the seminal paper \cite{bensoussantemam1973equations}. Here we analyze the case when $f$ is formally the time derivative of a cylindrical Wiener process $W$ in $H$ (see Section \ref{SEC:brownian}), appropriately regularized by an injective and bounded linear operator $G:H\to H$. 
These modifications turn the equation into a stochastic differential equation in the infinite dimensional Hilbert space $H$
\begin{equation}\label{EQ:protoNS}
    \begin{cases}
        \d u+[Au+B(u)]\d t =G\d W_t\qquad \text{ for }t>0\\
        u(0)=x\in H
    \end{cases}.
\end{equation}

We point out that two special types of noise have been much investigated in the literature: the case of space-time Gaussian white noise (\textit{i.e.} $G=\id^{}_H$, see, for instance, \cite{dapratodebussche2002two,GubinelliJara2013Regularization,albeverio2004uniqueness,ferrario20192dnavier} and the references therein); and the case of finite dimensional noise (\textit{i.e.} finite dimensional range of $G$, see, for instance, \cite{hairer2006ergodicity,hairer2011atheory,Kupiainen2011Ergodicity}). 
We  discuss more about cylindrical and finite dimensional noises  in Remark \ref{REM:hairer}. In this paper we concentrate on a case which is close to the cylindrical one, \textit{i.e.} we assume that  
$\Ran(G)\subset D(A^{\frac14+\eps})$ for some $\eps>0$ (\textit{cf.}  
 \cite{flandoli1994dissipativity, FlandoliMaslowski1995ergodicity, ferrario1997ergodic}).

 Interesting topics related to equation \eqref{EQ:protoNS} include defining a suitable notion of solution that assures both its existence and pathwise uniqueness. Once the solution is proven to be a Markov process, one associates with it a Markov semigroup (see Section \ref{SEC:P_tmu}), which formally describes the mean behavior in time of the solution given the starting point. This leads to other lines of research, including the regularity of the semigroup and the existence, uniqueness, and strong mixing property of the invariant measure $\mu$.
 The strong mixing property says that the law of the solution with a random starting point converges for long times to $\mu$ in the total variation norm (see Theorem \ref{TH:ergo}). This in particular implies the ergodicity of the invariant measure (\textit{cf.} \cite{da2014stochastic,seidler1997ergodic}), which means that $\mu$ is the equilibrium measure over the phase space $H$. This principle, known as the ergodic principle, forms the basis of the statistical approach to fluid dynamics.

Concerning pathwise uniqueness,  it can be proved for $\mu-a.e.$ initial data in the paper \cite{albeverio2004uniqueness} with $G=\id^{}_H$.
On the other hand, in our more regular setting, the problem of pathwise uniqueness for the solution (see Definition \ref{DEF:gensol}) has been solved in two ways with the assumption of $\Ran(G)\subset D(A^{\frac14+\eps})$ for some $\eps>0$ (see Remark \ref{REM:knownresult} \ref{IT:uniqu}). In \cite{flandoli1994dissipativity, FlandoliMaslowski1995ergodicity, ferrario1997ergodic}  a pathwise unique solution $u$ having paths in $C\big([0,T];H\big)\cap L^2\big(0,T;D(A^{\frac14})\big)$ is constructed such that $u-z\in L^2\big(0,T;D(A^{\frac12})\big)$, where $z$ is the Ornstein-Uhlenbeck process associated to the noise $G\d W$ (see Section \ref{SEC:brownian}).
More generally, the work \cite{ferrario2003uniqueness} gives a pathwise unique solution $u\in C\big([0,T];H\big)\cap L^4\big(0,T;D(A^{\frac14})\big)$.

Concerning path regularity, the assumption $\Ran(G)\subset D(A^{\frac14+\eps})$ for some $\eps>0$ has been investigated in \cite{flandoli1994dissipativity}, improving results by \cite{bensoussantemam1973equations} and \cite{vishik1988mathematical}.  
If one requires the stronger condition $\Ran(G)\subset D(A^{\frac38+\eps})$ for some $\eps>0$, then \cite{FlandoliMaslowski1995ergodicity} proves that the solution is more regular: indeed if $x\in D(A^{\frac14})$, then one gets $u\in C\big([0,T];D(A^{\frac14})\big)$ (see also Remark \ref{REM:knownresult} \ref{IT:G}). See also \cite{ferrario1997ergodic}.

We show in Section \ref{SEC:Mainresult} that under the assumptions $D(A^{\frac12})\subset \Ran(G)\subset D(A^{\frac14+\eps})$ for some $\eps\in(0,1/4]$ and $x\in D(A^\g)$ for some $\gamma\in[0,1/4+ \eps )$, the trajectories of the solution are continuous with values in $D(A^{\g})$ (see Theorem \ref{TH:mainresult}). 
This regularity result is obtained by \textit{a priori} estimates on the mild formulation through the Stokes semigroup for appropriate finite dimensional approximations (see in particular Lemma \ref{LEM:priola}). 
To obtain such estimates we are inspired by the technique introduced by Sobolevski\u{\i} in \cite{Soboleskii1959onnonstationary} and studied by Kato, Fujita and Giga (\textit{cf.} \cite{katofujita1962nonstationary, katofujita1964navier, giga1981analyticity, giga1983weak, gigamiyakawa1985solutions, giga1986solutions}). We refer to the introduction of \cite{kielhofer1980global} for a clever presentation of the Sobolevski\u{\i} approach. 
Comparing the semigroup approach initiated by Sobolevski\u{\i} to the classical energy estimate approach, we can say that, while the energy estimate for solutions is fundamental to prove that there is a global weak solution, every method has advantages and disadvantages. 
If we discuss the existence of a unique local strong solution, the semigroup method seems to be more powerful than the energy estimates and allows to achieve more regularity. 
To the best of our knowledge, previous papers on the stochastic N-S equations do not use this semigroup method based on the mild representation to establish further regularity of the solution.

In Section \ref{SEC:misura} we use the new path regularity result to study the problem of uniqueness and ergodicity of the invariant measure, which is known to exist if $\Ran(G)\subset D(A^{\frac14+\eps})$ for any $\eps>0$ (\textit{cf.} \cite{flandoli1994dissipativity}). 
To this aim, we also prove irreducibility and strong Feller properties for the Markov semigroup. 
We blend and adapt the reasoning followed in \cite{FlandoliMaslowski1995ergodicity} and \cite{ferrario1997ergodic} (see Remark \ref{REM:irredSF} \ref{IT:ISFcompare}), which managed to prove uniqueness of invariant measure and related ergodic properties under the stronger assumption $D(A^{\frac12})\subset \Ran(G)\subset D(A^{\frac38+\eps})$, for $\eps\in(0,1/8]$. 
This upper bound $\Ran(G)\subset D(A^{\frac38+\eps})$ appears also in more recent works, see for example \cite{dong2011ergodicity, goldys2005exponential}. 
We obtain uniqueness, ergodicity and a strong mixing property for the invariant measure in the hypothesis $D(A^{\frac12})\subset \Ran(G)\subset D(A^{\frac14})$ for $\eps\in(0,1/4]$. The case $\Ran(G)\subset D(A^{\frac12})$ has been studied in \cite{ferrario1999stochastic}.

\begin{remark}\label{REM:hairer} 
We recall that \cite{albeverio2004uniqueness} treats the stochastic N-S equations with space-time Gaussian white
noise, having a Gaussian infinitesimal invariant measure $\mu_\nu$ whose covariance is given in terms of the entropy. Pathwise
uniqueness for $\mu_\nu-a.s.$ initial velocity is proven for solutions having $\mu_\nu$ as
invariant measure.

The paper \cite{hairer2006ergodicity} provides uniqueness of invariant measure and ergodicity  by considering only finite dimensional noises. We believe that this powerful  approach based on the concept of asymptotic strong Feller property could be adapted also to our situation thanks to the  new regularity results in Section \ref{SEC:Mainresult}.
However, once we have proved Theorem \ref{TH:mainresult} it seems simpler to follow the arguments in \cite{FlandoliMaslowski1995ergodicity} and \cite{ferrario1997ergodic} to obtain results on the invariant measure (see Theorem \ref{TH:ergo}). 
Note, also, that the method in \cite{hairer2006ergodicity} does not provide the mutual equivalence to the laws of the solution, nor the strong mixing property we obtained. 
\end{remark}
\newpage

\section{Preliminaries}
\subsection{Functional setting}\label{SEC:functionalsetting}
Let $\mathcal D\subset \R^2$ be a non-empty, bounded and connected open set with Lipschitz boundary, then we denote by $C_c^\infty (\mathcal D;\R^2)$ the linear space of smooth vector fields with compact support and define the following spaces:
\begin{align*}
    H&:=\overline{\{u\in C_c^\infty(\mathcal D;\R^2) \ | \ \div u=0\}}^{L^2(\mathcal D;\R^2)},\\
    V&:=\overline{\{u\in C_c^\infty (\mathcal D;\R^2) \ | \ \div u=0\}}^{H^1(\mathcal D;\R^2)},\\
    \mathcal H^2&:=\overline{\{u\in C_c^\infty (\mathcal D;\R^2) \ | \ \div u=0\}}^{H^2(\mathcal D;\R^2)}.
\end{align*}
They inherit the Hilbert and embedding properties respectively from $L^2(\mathcal D;\R^2)$ and from the classical Sobolev spaces $H^1_0(\mathcal D;\R^2)$ and $H^2_0(\mathcal D;\R^2)$. 

We denote the inner product in $H$  by
\begin{equation*}
{\langle x, y\rangle}:=\int_{\mathcal D}x\cdot y\, \d\mathscr L^2=\sum_{i=1}^2\int_{\mathcal D}x_iy_i\d \mathscr L^2\qquad\forall\,x,y\in H,
\end{equation*}
where we integrate with respect to the two-dimensional Lebesgue measure $\mathscr L^2$, we denote by $x_i,y_i$ for $i=1,2$ the scalar components of $x,y\in H\subset L^2({\mathcal D};\R^2)$ and by $\cdot$ the euclidean inner product in $\R^2$. The norm induced by $\langle\,\cdot\,,\,\cdot\,\rangle$ is denoted by $\|\cdot\|$, while we reserve the symbol $|\cdot|$ for the euclidean norm in $\R^2$. We use the symbol $\prescript{}{V'}{\langle}\,\cdot\,,\,\cdot\,{\rangle}_{V}$ for the duality pairing between $V'$ and $V$. We denote the Hilbert norm on $V$ by
\begin{equation*}
\|x\|_V^{2}:=\sum_{j=1}^2\|\partial_jx\|^2=\sum_{i,j=1}^2\int_{\mathcal D}(\partial_jx_i)^2\d\mathscr L^2\qquad\forall\,x\in V,
\end{equation*}
where $\partial_j$ is the partial derivative (in the weak sense) with respect to the $j$-th variable in ${\mathcal D}\subset \R^2$.

\subsection{Operators}\label{SEC:operators}
Since $H$ is a Hilbert subspace of $L^2({\mathcal D};\R^2)$ we can construct an orthogonal projector operator, commonly called Leray projector and denoted by $\Pi:L^2(\mathcal D;\R^2)\to H$. We define the Stokes operator as
\[
A:\mathcal H^2\to H:x\mapsto 
\Pi\begin{pmatrix}
-\Delta x_1\\
-\Delta x_2
\end{pmatrix},
\]
where $\Delta$ is the weak Laplacian operator on $H_0^2(\mathcal D)$. Here $A$ is a linear, bounded, positive-definite, self-adjoint, invertible operator. Its inverse $A^{-1}$ is a compact operator on $H$, therefore there exists an orthonormal complete system ${\{e_k\}}_{k\in\mN}$ in $H$ made of eigenvectors of $A$ and a strictly increasing and diverging sequence of eigenvalues ${\{\lambda_k\}}_{k\in\mN}\subset (0,+\infty)$ such that $Ae_k=\lambda_k e_k$ for all $k\in\mN$. If we endow the vector space $\mathcal H^2=D(A)$ with the norm induced from $H$, then $-A:D(A)\subset H\to H$ is an unbounded, closed, densely defined and self-adjoint operator (see \cite[Theorem $3$]{Fujiwaramorimoto1977Helmholtz}) which generates a one-parameter analytic semigroup of linear bounded operators ${\{e^{-tA}\}}_{t\geq 0}$ (see \cite[theorem $2$]{giga1981analyticity}). 

For any $\alpha>0$ we define in the usual way the injective and bounded linear operator $A^{-\a}$ on $H$. It is known that, in $2$-space dimensions, $\la_k$ is asymptotic to $ck$ for a certain $c>0$ as $k\to\infty$ (\textit{cf.} \cite{metivier1978valeurspropres,caetano1998eigenvalue}), thus $A^{-\a}$ is Hilbert-Schmidt if $\a>1/2$ and trace-class if $\a>1$. 
The space of linear bounded operators on $H$ will be denoted by $\mathcal L(H)$, while its subspace of Hilbert-Schmidt operators by $\mathcal L_2(H)$, both endowed with their usual complete norm. 
For any $\a>0$ we denote by $D(A^\a)\subset H$ the range of $A^{-\a}$ and by $A^\a:D(A^\a)\subset H\to H$ the unbounded inverse operator of $A^{-\a}:H\to D(A^\a)$ (\textit{cf.} \cite[Chapters $1,2$]{pazy1992semigroups}). 
It is known (\textit{cf.} \cite[Theorem $1.1$]{fujitamorimoto1970fractional}) that
\[
D(A^{\a})=\overline{\{u\in C_c^\infty (\mathcal D;\R^2) \ | \ \div u=0\}}^{H^{2\a}(\mathcal D;\R^2)}\qquad\forall\,\a>0,
\]
where $H^{2\a}(\mathcal D;\R^2)$ is the classical fractional Sobolev space (see, for instance, \cite{dinezza2012hitchhikers}).
Moreover the complete norm on $D(A^\a)$ given by 
\[
\|x\|^{}_{D(A^\a)}:=\|A^\a x\|\qquad \forall\, x\in D(A^\a),
\]
is equivalent to the usual Sobolev norm on ${H_0^{2\a}(\mathcal D;\R^2)}$. We identify $H$ with its topological dual $H'$, thus we have the following compact and dense embeddings for all $\a\in(0,1/2)$ and $\b>1/2$:
\begin{equation*}
D(A^{\b})\xhookrightarrow{}V=D(A^{\frac12})\xhookrightarrow{} D(A^\alpha)\xhookrightarrow{}H=H'\xhookrightarrow{}D(A^\a)'\xhookrightarrow{}V'=D(A^{\frac12})'\xhookrightarrow{}D(A^\b)'.
\end{equation*}

Since $A^{-\a}$ is bounded and positive, it is also self-adjoint, thus the inverse $A^\a:D(A^\alpha)\subset H\to H$ is a self-adjoint unbounded operator (\textit{cf.} \cite[Proposition $8.2$]{taylor1996partial}). Moreover $A^\a e^{-tA}$ is bounded for all $t>0$ and the following property holds (see \cite[Lemma $2.10$]{katofujita1964navier})
\begin{gather}\label{EQ:Aaet1}
    \|A^\alpha e^{-tA}\|_{\mathcal L(H)}
    \leq \Big(\frac{\a}{e}\Big)^{\a}t^{-\a}\qquad\forall \, t> 0, \  \forall\,  \a> 0.    
\end{gather}

From the properties of $A$ and the H\"older inequality, we derive the following lemma. 
\begin{lemma}[Interpolation inequality]\label{LEM:sobolevinterpolation} For all $0\leq p< q<+\infty$ and for all $u\in D(A^q)$ it holds
    \begin{equation*}
     \|A^ru\|\leq \|A^pu\|^\la \|A^qu\|^{1-\la} \qquad  \forall\, \la\in(0,1), \ r:=\la p + (1-\la)q.
     \end{equation*}
\end{lemma}

Eventually, we define the function  
\[
B:C_c^\infty ({\mathcal D};\R^2)\to H:x\mapsto \Pi\big[(Dx)x\big]=\Pi\begin{pmatrix}
x\cdot \nabla x_1\\
x\cdot\nabla x_2
\end{pmatrix},
\]
which is commonly known as the Navier-Stokes non-linearity. Similarly, we introduce the trilinear bounded operator
\begin{equation*}
 b :V\times V\times V\to \R : (x,y,z)\mapsto \int_{{\mathcal D}}\begin{pmatrix}
x\cdot \nabla y_1\\
x\cdot\nabla y_2
\end{pmatrix}\cdot z\, \d \mathscr{L}^2=\sum_{i,j=1}^2\int_{{\mathcal D}}z_j\,x_i\,\partial_iy_j\, \d \mathscr L^2,
\end{equation*}
which is antisymmetric upon exchange of the second and third entries, thanks to integration by parts and the fact that vector fields in $V$ have vanishing divergence. The following result is taken from \cite[Lemma $3.2$]{giga1983weak} and characterizes the extensions of $b$ and $B$.

\begin{lemma}\label{LEM:bB}
    For all $\delta\in[0,1)$ and $\theta,\rho>0$ such that $\rho+\theta+\delta\geq 1$ and $\rho+\delta>1/2$, $ b $ can be uniquely extended to a bounded trilinear operator and $B$ to a continuous function as follows
    \begin{gather*}
     b :D(A^{\theta})\times D(A^{\rho})\times D(A^{\delta})\to\R,\qquad
     B:D(A^{\rho\vee \theta})\to D(A^\delta)'.
    \end{gather*}
    Moreover, there exists a constant $c_0=c_0(\theta,\rho,\delta)>0$ such that 
    \begin{gather*}
    |b(x,y,z)|\leq c_0\|A^\theta x\| \|A^\rho y\| \|A^\delta z\|,\qquad
    \|A^{-\delta}B(x)\|\leq c_0\|A^\rho x \| \|A^{\theta}x\|.
    \end{gather*}
    Eventually, as soon as both members make sense, it holds $ b (x,y,z)=- b (x,z,y)$.
\end{lemma}
    \noindent We observe in particular that $B$ allows to rewrite
         $b (x,x,z)=
        \prescript{}{ V'}{\langle} B(x),z{\rangle}_{ V}$ for all $x\in D(A^{\frac14})$ and $z\in V$.

\subsection{Brownian noise}\label{SEC:brownian}
We consider a filtered and complete probability space $\big(\Omega,\mathcal F, {\{\mathcal F_t\}}_{t\geq 0},\mP\big)$, a sequence of mutually independent ${\{\mathcal F_t\}}_{t\geq 0}$-adapted real Brownian motions ${\{w^k\}}_{k\in\mN}$ and we introduce a coloured Wiener noise as 
\begin{equation}\label{EQ:defQW}
GW_t:=\sum_{k\in\mN}w^k_tGe_k\qquad \forall\, t\geq 0,
\end{equation}
where $\{e_k\}_{k\in\mN}$ are the eigenvectors of $A$ and $G\in\mathcal L(H)$ is an injective and bounded linear operator.
We observe that if $G$ is Hilbert-Schmidt, $GW$ is an $H$-valued Wiener process, thus the series in equation \eqref{EQ:defQW} converges for all $t \geq 0$ in $L^2(\Omega;H)$ or $\mP-a.s.$ in $C\big([0,T];H\big)$ for all $T > 0$ (see \cite[Theorem $4.5$]{da2014stochastic}).
However, our primary interest lies in more general noises: 
if $G$ is not Hilbert-Schmidt, $GW$ is defined as a generalized Wiener process (\textit{cf.} \cite[Section $4.1.2$]{da2014stochastic}). 
The main assumptions on $G$ that will be used throughout the paper are the following
\begin{align}\label{HP:H0}
&G\in\mathcal L(H),\qquad G\text{ injective,} \qquad \exists \, \eps>0 \quad \text{s.t.} \quad \Ran(G)\subset D(A^{\frac 14+\eps})\tag{$H_0$},\\
&G\in\mathcal L(H),\qquad G\text{ injective,} \qquad \exists \, \eps\in(0,1/4] \quad \text{s.t.} \quad V\subset \Ran(G)\subset D(A^{\frac 14+\eps}).\tag{$H_1$}\label{HP:H1}
\end{align}
Clearly \eqref{HP:H1}$\Longrightarrow$\eqref{HP:H0}. Hypothesis \eqref{HP:H0} is the basic assumption found in literature (\textit{cf.} \cite{FlandoliMaslowski1995ergodicity,ferrario1997ergodic}) that guarantees uniqueness of generalized solution and existence of invariant measure for our equation (see Sections \ref{SEC:equation} and \ref{SEC:P_tmu}). Conversely, \eqref{HP:H1} is the required hypothesis for our setting.

We remark that $
\langle GW_t,\phi\rangle=
\sum_{k\in\mN}w^k_t\langle Ge_k,\phi\rangle$ for all $t\geq 0$ and $\phi\in D(A)$,
where the series converges for all $t\geq 0$ in $L^2(\Omega)$ or almost surely in $C([0,T])$ for all $T>0$.

The following theorem defines and characterizes a stochastic process ${\{Z_t\}}_{t\geq 0}$, which will be often referred to as stochastic convolution or Ornstein-Uhlenbeck process starting at $0\in H$.

\begin{theorem}\label{TH:z}
Let $\eps,G$ satisfy assumption \eqref{HP:H0}. There is a pathwise unique predictable process ${\{Z_t\}}_{t\geq 0}$ with $\mP-a.s.$ trajectories $z\in L^1(0,T;H)$ for all $T>0$ such that 
\begin{equation*}
\langle z(t), \phi\rangle +\int_0^t\langle z(s), A\phi\rangle\d s =\langle GW_t,\phi\rangle \qquad \forall \, \phi\in  D(A), \ \forall \, t\geq 0, \ \mP-a.s.
\end{equation*}
Moreover $z\in C^\b\big([0,T];D(A^\g)\big)$ for any $\b>0$ and $\g\geq 0$ such that $\b+\g<(1/4+\eps)\wedge 1/2$, for all $T>0$ and $\mP-a.s.$
\end{theorem}
\begin{proof}
    For the existence and uniqueness result we refer to \cite[Theorem $5.4$]{da2014stochastic}. The second statement for $\eps\in(0,1/4]$ descends from \cite[Theorem $5.15$]{da2014stochastic}, once checked that for any $\a<1/4+\eps$, it holds
    \begin{equation*}
    \int_0^1t^{-2\a}\|e^{-tA}G\|^2_{\mathcal L_2(H)}\d t
    =\int_0^1t^{-2\a}\operatorname{Tr}\big[e^{-tA}GG^*e^{-tA}\big]\d t<+\infty.
    \end{equation*}
    Let $p=1/4+\eps$; we have
    \begin{align*}
    \int_0^1t^{-2\a}\|e^{-tA}G\|^2_{\mathcal L_2(H)}\d t &
    =\int_0^1t^{-2\a}\|e^{-tA}A^{-p}A^{p}G\|^2_{\mathcal L_2(H)}\d t
    \leq \|A^pG\|^2_{\mathcal L(H)}\int_0^1t^{-2\a} \|e^{-tA}A^{-p}\|^2_{\mathcal L_2(H)}  \d t,
    \end{align*}
    where we used the properties of the Hilbert-Schmidt norm and the fact that $A^pG\in\mathcal L(H)$, thanks to hypothesis \eqref{HP:H0}. In order to compute the Hilbert-Schmidt norm of $e^{-tA}A^{-p}$, we resort to the orthonormal complete system $\{e_k\}_{k\in\mN}^{}$ for $H$ made of eigenvectors of $A$
    \[
    \int_0^1t^{-2\a}\|e^{-tA}A^{-p}\|^2_{\mathcal L_2(H)}\d t =\sum_{k=1}^\infty\la_k^{-2p}\int_0^1t^{-2\a} e^{-2t\la_k}\d t 
    \leq \sum_{k=1}^\infty\la_k^{-2p-1+2\a}\int_0^{+\infty}e^{-2s}s^{-2\a}\d s.
    \]
    The time integral converges for all $\a<1/2$. It is known (\textit{cf.}  \cite{metivier1978valeurspropres}) that the Stokes eigenvalues $\la_k$ in our two-dimensional setting are asymptotic to $ck$ for some constant $c>0$, as $k\to\infty$. Consequently the series converges for all $\a<p=1/4+\eps$. An application of \cite[Theorem $5.15$]{da2014stochastic} gives $Z$ path regularity $C^\b\big([0,T];D(A^\g)\big)$ for any $\b+\g<\a$ for any $\a<1/4+\eps$, thus also for any $\b+\g<1/4+\eps$. Finally it is shown in \cite[Section $5.4.2$]{da2014stochastic} shows that we can't choose $\g=1/2$, thus preventing the continuity of the trajectories in $V=D(A^{\frac12})$, even if $\eps>1/4$.
\end{proof} 

\begin{remark}
    Theorem \ref{TH:z} underscores the significance of the hypothesis \eqref{HP:H0} on the range of $G$. It sets our minimal regularity for the noise, thus effecting the solution of the N-S equations, as will be discussed in the next section. For $\Ran(G)\subset D(A^{\frac14+\eps})$ with $\eps$ arbitrarily close to $0$, we reach continuity in time with values in $D(A^{\frac14})$. On the other hand, for $\Ran(G)\subset D(A^{\frac12})$, $Z$ achieves the maximal regularity as a $\mP-a.s.$ continuous process with values in $D(A^\g)$ for $\g>0$ arbitrarily close to $1/2$.
\end{remark}

\begin{example}
An example of operator $G$ that satisfies our assumptions \eqref{HP:H1} is given by
\[
G=A^{-\g}L,
\]
with $\g\in(1/4,1/2]$ and $L$ an isomorphism in $H$.
    Another example is
    \[
    Gx=\sum_{k\in\mN}\frac{\langle x,e_k\rangle}{\sigma_k}e_k\qquad\forall\, x\in H,
    \]
    where $ak^{1/4+\eps}\leq \sigma_k\leq  bk^{1/2}$ for some constants $a,b>0$ and all $k\in\mN$. Recall indeed that the eigenvalues $\lambda_k$ of the Stokes
operator $A$ behave like $ck$, for some $c>0$, as $k\to\infty$ (\textit{cf.} \cite{metivier1978valeurspropres}).
\end{example}

\subsection{Abstract equation}\label{SEC:equation}
We study the abstract 2D N-S stochastic equations in $H$, for some starting point $x\in H$ and some bounded linear operator $G$ satisfying assumption \eqref{HP:H0}
\begin{equation}\label{EQ:NS}
\begin{cases}
\d X^x_t+[AX^x_t+B(X^x_t)]\d t=G\d W_t\qquad & t>0,\, \mP-a.s. \\
X^x_0=x  &\mP-a.s.
\end{cases}.
\end{equation}
We state the definition of solution we employ in the paper, which is taken from \cite{flandoli1994dissipativity}.
\begin{definition}\label{DEF:gensol}
    Given $x\in H$ and $G$ as in hypothesis \eqref{HP:H0}, a generalized solution to equation \eqref{EQ:NS} is a progressively measurable process $X^x$ in $H$ with path regularity 
    \[
    X^x(\omega)=u\in C\big([0,T];H\big)\cap L^2\big(0,T; D(A^{\frac14})\big)\qquad \forall\, T>0, \ \mP-a.s.\ \omega\in\Omega
    \]
    such that 
    \begin{gather*}
        \langle u(t),\phi\rangle+\int_0^t\langle u(s),A\phi\rangle \d s=\int_0^t b\big(u(s),\phi, u(s)\big)  \d s  +\langle x,\phi\rangle+\langle GW_t, \phi\rangle\\
        \forall \, t\geq 0, \ \forall\, \phi\in D(A), \ \mP-a.s.
    \end{gather*} 
\end{definition}

The following theorem summarizes the results in the literature regarding the existence, uniqueness, and path regularity of the generalized solution.
\begin{theorem}\label{TH:knownresult}
For every $x\in H$ and $\eps,G$ as in hypothesis \eqref{HP:H0}, there exists a pathwise unique generalized solution to equation \eqref{EQ:NS} (\textit{cf.} Definition \ref{DEF:gensol}) with trajectories $u$ that satisfy $\mP-a.s.$ $u-z\in L^2(0,T;V)$ for all $T>0$, where $z$ is given by Theorem \ref{TH:z}. Moreover this solution is a Markov process in $H$ and satisfies $\mP-a.s.$ the following additional path regularities: $u\in L^2\big(0,T;D(A^{(\frac14+ \eps )\wedge\frac12}\big)\cap L^4\big(0,T; D(A^{\frac14})\big)$ for all $T>0$.
\end{theorem}
\begin{proof}
    The first assertions follow by \cite[Theorem $3.1$]{flandoli1994dissipativity}. In order to show the additional path regularity, we simply observe that the space $L^2(0,T;V)\cap C\big([0,T];H\big)$ is continuously embedded into $L^4\big(0,T; D(A^{\frac14})\big)$. Indeed we have for a generic $v\in L^2(0,T;V)\cap C\big([0,T];H\big)$:
    \[\int_0^T\|A^{\frac14}v(t)\|^4\d t = 
    \int_0^T\langle A^{\frac12}v(t),v(t)\rangle ^2\d t 
    \leq \int_0^T\|v(t)\|^2_V\|v(t)\|^2\d t 
    \leq \|v\|_{C([0,T];H)}^2\|v\|^2_{L^2(0,T;V)}.\]
    Therefore if $u$ is the unique solution by \cite[Theorem $3.1$]{flandoli1994dissipativity}, then $u-z\in L^4\big(0,T; D(A^{\frac14})\big)$ and consequently $u=(u-z)+z\in L^4\big(0,T; D(A^{\frac14})\big)$, thanks to Theorem \ref{TH:z}.
\end{proof}
\begin{remark}\label{REM:knownresult}
    \begin{enumerate}[wide, label=$(\roman*)$]
    \item \label{IT:uniqu} In \cite{ferrario2003uniqueness} it is proved that a wider class that assures a pathwise unique generalized solution is $C\big([0,T];H\big)\cap L^4\big(0,T; D(A^{\frac14})\big)$, to which belongs our unique solution from Theorem \ref{TH:knownresult}. Indeed, the unique solution by Ferrario may not satisfy the additional condition $u-z\in L^2(0,T;V)$, which is used in \cite{flandoli1994dissipativity}.
        \item The technique employed in \cite{flandoli1994dissipativity} to construct a generalized solution combines the general approach to stochastic partial differential equations with additive noise (see, for instance, \cite[Chapter $7$]{da2014stochastic}) with the classical energy estimates on the Galerkin approximations used for the deterministic case (see, for instance, \cite[Chapter $3$]{temam2001navier}). 
        \item \label{IT:G}
        While for injective $G\in\mathcal L(H)$ with $\Ran(G)\subset D(A^{\frac12+\eps})$ for some $\eps>0$, the maximal regularity in space $u\in L^2(0,T;V)$ follows by 
        the deterministic N-S equations; if $G$ is more degenerate, in the sense that $\Ran(G)\subset D(A^{\frac14+\eps})$ with $\eps\in(0,1/4]$ (\textit{i.e.} assumption \eqref{HP:H1}), then the space regularity $u\in L^2\big(0,T;D(A^{\frac14+ \eps })\big)$ is inherited by the stochastic convolution $Z$ (see Theorem \ref{TH:z}). 
        In other words, the wider is the range of $G$, the less regular are the trajectories of $Z$, thus resulting in worse path regularities for the generalized solution. Conversely, when the range of $G$ narrows (but remains not smaller than $V=D(A^{\frac12})$), both $Z$ and $X$ exhibit better path regularities. Ultimately, if the range of $G$ is contained within $V=D(A^{\frac12})$, the generalized solution achieves maximal path regularity. \\        
        Incidentally, under the stronger assumptions $\Ran(G)\subset D(A^{\frac38+\eps})$ with $\eps>0$ and $x \in  D(A^{\frac14})$, \cite[Theorem $2.1$]{FlandoliMaslowski1995ergodicity}
        claims that $u \in C\big([0,T]; D(A^{\frac14})\big)\cap L^2\big(0,T;D(A^{\frac38})\big)$. Under these conditions, the generalized solution was proven to be pathwise unique in the set $C\big([0,T]; D(A^{\frac14})\big)\cap L^8\big(0,T;D(A^{\frac38})\big)$ (\textit{cf.} \cite[Theorem $4.2$]{ferrario2003uniqueness}).
    \end{enumerate}
    \end{remark}

\subsection{Invariant measure}\label{SEC:P_tmu}
Let $ \mathcal B_b(H)$ be the linear space of all Borel and bounded functions $\varphi:H\to\R$, endowed with the complete norm ${\|\varphi\|}_\infty:=\sup_{y\in H}|\varphi(y)|$. According to Theorem \ref{TH:knownresult}, we can define the Markov semigroup $P=\{P_t\}^{}_{t\geq 0}$ associated to the generalized solution of equation \eqref{EQ:NS} as follows
\[
P_t\varphi(x):=\mE\, \varphi( X^x_t)\qquad \forall \, t\geq 0, \ \forall \, \varphi\in  \mathcal B_b(H), \ \forall \, x\in H.
\]

Let $\mathcal P(H)$ be the set of all probability measures over the Borel sigma-algebra $\mathscr B_H$ on $H$, then we can define 
\[
P^*_t\nu(U):=\int_H\mP\big(X^x_t\in U\big) \,\nu(\!\d x)\qquad \forall \, t\geq 0, \ \forall \, \nu\in\mathcal P(H), \ \forall\,  U\in\mathscr B_H,
\]
and it is readily verified that $P^*_t\nu$ is again a probability measure on $\mathscr B_H$. 

A probability measure $\mu\in\mathcal P(H)$ is said to be an invariant measure for the semigroup $\{P_t{\}}_{t\geq 0}$ if $P^*_t\mu=\mu$ for all times $t\geq 0$ (see also \cite{brzezniak2022ergodicity} and the references therein).
The following theorem collects known results from literature regarding existence and uniqueness of the invariant measure (see \cite[Theorem $3.3$]{flandoli1994dissipativity}, \cite[Theorem $3.1$]{FlandoliMaslowski1995ergodicity} and  \cite[Corollary $4.1$]{ferrario1997ergodic}).
\begin{theorem}\label{TH:munoto}
    Under hypothesis \eqref{HP:H0} there exists an invariant measure $\mu$ for the Markov semigroup $\{P_t\}^{}_{t\geq0}$ associated with the generalized solution to equation \eqref{EQ:NS}. Moreover, under the stronger hypotheses
    \begin{align}\label{HP:H2}
G\in\mathcal L(H),\qquad G\text{ injective,} \qquad \exists \, \eps\in(0,1/8] \quad \text{s.t.} \quad V\subset \Ran(G)\subset D(A^{\frac 38+\eps})\tag{$H_2$}
\end{align}
$\mu$ is known to be unique in the set $\mathcal P(H)$ and concentrated on the Borel set $ D(A^{\frac14})\subset H$. 
\end{theorem}
Note that \eqref{HP:H2}$\Longrightarrow$\eqref{HP:H1}$\Longrightarrow$\eqref{HP:H0}. In Section \ref{SEC:misura} we will generalize the uniqueness and concentration results of this theorem to the weaker assumption \eqref{HP:H1}.

\section{Main result}\label{SEC:Mainresult}
This section is devoted to the following main result regarding the additional path regularity we obtained for the unique generalized solution.
\begin{theorem}\label{TH:mainresult}
    Let $\eps,G$ be as in hypothesis \eqref{HP:H1}. For every $\gamma\in[0,1/4+ \eps )$ and every starting point $x\in D(A^{\g})$ the unique generalized solution to equation \eqref{EQ:NS} from Theorem \ref{TH:knownresult} has the additional regularity
    \[
    u\in C\big([0,T];D(A^\g)\big)\qquad \forall \, T>0, \ \mP-a.s.
    \]
    Moreover, if $x\in H$ then $
    u\in C\big([t_0,T];D(A^\g)\big)$ for all $0<t_0<T$ and $\mP-a.s.$
\end{theorem}

The proof of this theorem employs the Sobolevski\u{\i}-Kato-Fujita method, which involves the mild formulation to the N-S equations (\textit{cf.} \cite{Soboleskii1959onnonstationary, katofujita1962nonstationary}). We outline the approach to be followed in the subsequent subsections.

\begin{enumerate}
\item First, we utilize a known technique to investigate abstract stochastic partial differential equations with additive noise (see, for instance, \cite[Chapter $7$]{da2014stochastic}).  This involves fixing $\mP-a.s.$ $\omega\in\Omega$  and formally introducing the equation satisfied by $v=u-z$, where $z$ is the fixed trajectory of the Ornstein-Uhlenbeck process (\textit{cf.} Theorem \ref{TH:z}):
    \begin{equation*}
         \begin{cases}
             v'+Av+B(v+z) =0 \qquad t>0\\
             v(0)=x
         \end{cases}.
     \end{equation*}
We rigorously study this equation by means of finite dimensional approximations $v_n$ and obtain some \textit{a priori} estimates (see Lemma \ref{LEM:atimaprioriferrario}), similar to those found in \cite{ferrario1997ergodic}.
 
\item We rewrite $v_n$ through the mild formulation and obtain a new \textit{a priori} bound in $L^\infty\big(0,T;D(A^{\frac14})\big)$, arguing as in the Sobolevski\u{\i}-Kato-Fujita approach (slightly more spatial regularity will be obtained, see Lemma \ref{LEM:priola}).

\item This new estimate allows us to establish the continuity of $v$ through the Arzelà-Ascoli theorem. Subsequently, we define $u_n:=v_n+z$ and infer convergence to the unique $u$ given in  Theorem \ref{TH:knownresult} in appropriate function spaces (see Theorem \ref{TH:ulunardi}). 
\end{enumerate}

\subsection{Finite dimensional approximations for \texorpdfstring{$v=u-z$}{Lg}}
For all $n\in\mN$ let $\Pi_n$ be the projector onto the finite dimensional subspace of $H$ generated by the first $n$ vectors of the complete orthonormal system $\{e_k\}^{}_{k\in\mN}$ (seee Section \ref{SEC:operators}). We denote $H_n=\Pi_nH$, $B_n:H_n\to H_n:x\mapsto \Pi_nB(x)$ and  $x_n= \Pi_nx$, for any $x\in H$.
 Let $z$ be a $\mP-a.s.$ fixed trajectory of the stochastic convolution (see Theorem \ref{TH:z}), then we study the following equation in finite dimensions over the time interval $[0,T]$ for an arbitrarily fixed $T>0$:
 \begin{equation}\label{EQ:approximatedvn}
 \begin{cases}
 v'_n+Av_n +B_n\big(v_n+z\big) =0 \qquad t\in(0,T]\\
 v_n(0)=x_n
 \end{cases}.
 \end{equation}
 We know by the theory of ODEs that there exists a pathwise unique $v_n\in C\big([0,T];H_n\big)$ such that 
 \begin{equation}\label{EQ:v_nVolterra}
 v_n(t)+\int_0^tAv_n(s)\d s +\int_0^tB_n\big(v_n(s)+z(s)\big)\d s=x_n\qquad \forall \, t\in[0,T], \ \mP-a.s.
 \end{equation}
 We observe that in finite dimensions all the norms on $H_n$ are equivalent, thus $v_n(t)\in D(A^\a)$ for all $n\in\mN$, $\a\geq 0$, $t\geq0$ and $\mP-a.s.$ Moreover, the process $v_n$ has almost surely smooth paths. Therefore the equations in system \eqref{EQ:approximatedvn} are satisfied almost surely in probability and pointwise in time. 
 
 \begin{remark}\label{REM:v_ntov}
     Inspired by the classical reasoning in \cite{temam2001navier}, it is proved in \cite{flandoli1994dissipativity} that, under the assumption \eqref{HP:H0}, and for any $x\in H$ and $T>0$, there exists a sub-sequence of $v_n$ converging weakly* in $L^\infty(0,T;H)$, weakly in $L^2(0,T;V)$ and strongly in $L^2(0,T;H)$ to a function $v\in C\big([0,T];H\big)\cap L^2(0,T;V)$ which is the pathwise unique solution to equation
     \begin{equation}\label{EQ:systemv}
         \begin{cases}
             v'+Av+B(v+z) =0\qquad t\in(0,T]\\
             v(0)=x
         \end{cases}
     \end{equation}
     in the following generalized sense $\mP-a.s.$
 \begin{equation}\label{EQ:defv}
 \begin{split}
 \langle v(t), \phi\rangle +\int_0^t\langle v(s), A\phi\rangle \d s =\int_0^t b\big(v(s)+z(s), \phi, v(s)+z(s)\big)\d s+\langle x, \phi\rangle\\\forall \, \phi\in D(A), \  \forall \, t\in[0,T].
 \hspace{4cm}
 \end{split}
 \end{equation}
 We recall that this result does not depend on the stochastic properties of the Ornstein-Uhlenbeck process, but only on its path regularity $z\in C\big([0,T]; D(A^{\frac14})\big)$ for all $T>0$.
 \end{remark}

We henceforth replace $v_n$ with its converging subsequence, that we still denote as $v_n$.
We are going to obtain some \textit{a priori} estimates by adapting the calculations provided in the proof of \cite[Proposition $4.1$]{ferrario1997ergodic}.
 \begin{lemma}\label{LEM:atimaprioriferrario}
  Let $\eps,G$ be as in hypothesis \eqref{HP:H1}. For all $T>0$ and $p\in\big[4,4/(1-2\eps)\big)$ there exists a constant $c_1=c_1(\eps,G,p,T)>0$ such that
  \[
  \sup_{n\in\mN}\int_0^T\|A^{1/4
  }v_n(t)\|^{p}\d t \leq c_1\big(1+\|A^\eps x\|\big)\qquad \forall\,x\in D(A^\eps), \ \mP-a.s.
  \]
 \end{lemma}
 \begin{proof}
    We rename $2\a=1-4/p$ and we observe that the bounds on $p$ translate into the bounds $0\leq\a<\eps$. 
     Let us take the ODE in equation \eqref{EQ:approximatedvn}, which is satisfied almost surely in probability and pointwise in time, and take the scalar product in $H$ with $A^{2\a} v_n(t)$. By the sake of brevity we omit the dependence on $t$. 
     \begin{align*}
         \frac{1}{2}\frac{\d}{\d t} \|A^{\a}v_n\|^2+\|A^{\a+\frac12}v_n\|^2
         &=\langle v'_n,A^{2\a}v_n\rangle +\langle Av_n,A^{2\a} v_n\rangle \\
         &=-\langle A^{\a-\frac12}B_n(v_n+z), A^{\a+\frac12}v_n\rangle\\
         &\leq c_0\|A^{\frac14+\frac\a2}(v_n+z)\|^2\|A^{\a+\frac12}v_n\|\\
          &\leq 2c_0\Big(\|A^{\frac12}v_n\| \|A^\a v_n\|+\|A^{\frac14+\frac\a2}z\|^2\Big)\|A^{\a+\frac12}v_n\|\\
          &\leq 4c_0^2\|A^\a v_n\|^2\|A^{\frac12}v_n\|^2+4c_0^2\|A^{\frac14+\frac\a2}z\|^4+\frac{1}{2}\|A^{\a+\frac12}v_n\|^2
     \end{align*}
     To perform the estimate, we first applied the Cauchy-Schwarz inequality together with Lemma \ref{LEM:bB}, with the choices $\delta=1/2-\a$, $\theta=\rho=1/4+\a/2$. Next we used the Young inequality and the interpolation inequality (\textit{cf.} Lemma \ref{LEM:sobolevinterpolation}) with the choices $q=1/2, \la=1/2, r=1/4+\a/2$. Finally, we applied the Young inequality again. 
     
     If we rewrite the first and last member we reach
     \begin{align*}
         \frac{\d}{\d t} \|A^\a v_n\|^2\leq 
         \frac{\d}{\d t} \|A^\a v_n\|^2+\|A^{\a + \frac12}v_n\|^2\leq  
         8c_0^2\|A^\a v_n\|^2\|A^{\frac12}v_n\|^2+8c_0^2\|A^{\frac14+\frac\a2}z\|^4.
     \end{align*}
     We integrate over the time interval $[0,t]$
     \begin{align*}
         \|A^\a v_n(t)\|^2\leq \|A^\a x_n\|^2 +8c_0^2\|A^{\frac14+\frac\a2}z\|^4_{L^4(0,T;H)} + 
         8c_0^2\int_0^t\|A^\a v_n(s)\|^2\|A^{\frac12}v_n(s)\|^2\d s,
     \end{align*}
    and we apply Gr\"onwall's lemma:
\begin{align}\label{EQ:stimaaprioriv_n}
     \|A^\a v_n(t)\|^2&\leq 
     \Big(\| A^\a x_n\|^2+8c_0^2\|A^{\frac14+\frac\a2}z\|^4_{L^4(0,T;H)} \Big)\exp \bigg[8c_0^2\!\int_0^t\|A^{\frac12}v_n(s)\|^2\d s\bigg].
     \end{align}
     Since $\a<\eps$, as discussed at the beginning of the proof, we have $\|A^\a x_n\|=\|\Pi_nA^\a x\|\leq c\|A^\eps x\|$. Moreover, Theorem \ref{TH:z} implies that $\|A^{\frac14+\frac\a2} z\|_{L^4(0,T;H)}$ is almost surely bounded. 
     Eventually, we know that $v_n$ converges in $L^2(0,T;V)$ as $n$ approaches infinity to the function $v$ defined by equation \eqref{EQ:defv} (\textit{cf.} Remark \ref{REM:v_ntov}). Therefore $A^{\frac12}v_n$ is uniformly bounded in $L^2(0,T;V)$.
     To sum up, estimate \eqref{EQ:stimaaprioriv_n} results in the following \textit{a priori} bound for $\|A^\a v_n(t)\|$, uniform both in $n$ and $t$:
     \[
     \sup_{n\in\mN}\sup_{t\in[0,T]} \|A^\a v_n(t)\|\leq 
     C\big(1+\|A^\eps x\|\big)\qquad \mP-a.s.
     \]
     By means of the interpolation inequality (\textit{cf.} Lemma \ref{LEM:sobolevinterpolation}) with coefficients $q=1/2$, $\la=1-2/p$, $r=\la \a+(1-\la)q=1/2-4/p^2$ we have 
     \begin{align*}
     \|A^{\frac14}v_n(t)\| &\leq c\|A^{r}v_n(t)\| 
     \leq c\|A^\a v_n(t)\|^{\la}\|A^{\frac12}v_n(t)\|^{1-\la}
     \leq K\big(1+\|A^\eps x\|\big)^\la\|A^{\frac12}v_n(t)\|^{1-\la}.
     \end{align*}
     By raising to the power of $p=2/(1-\la)$ and integrating in time we reach the sought \textit{a priori} estimate, for a constant $c_1>0$ depending only on $\eps,G,p$ and $T$:
     \begin{equation*}
    \int_0^T\|A^{\frac14}v_n(t)\|^{p}\d t\leq 
      K^p\big(1+\|A^\eps x\|\big)\int_0^T\|A^{\frac12}v_n(t)\|^{2}\d t\leq c_1\big(1+\|A^\eps x\|\big) \qquad \forall \,n\in\mN, \ \mP-a.s.
     \end{equation*}
 \end{proof}

\subsection{Mild formulation}
We now shift to the mild formulation, thanks to the following standard lemma.
\begin{lemma}
     For any $n\in\mN$ the stochastic process $v_n$ satisfies the integral formulation in equation \eqref{EQ:v_nVolterra} if and only if it satisfies the following mild formulation
    \begin{equation}\label{EQ:v_nmild}
    v_n(t)+\int_0^te^{-(t-s)A}B_n\big(v_n(s)+z(s)\big)\d s =e^{-tA}x_n \qquad \forall \, t\geq 0, \ \mP-a.s.
\end{equation}
 \end{lemma}

Note that the mild formulation has been already used to study the stochastic N-S equations (see, for instance, \cite[Chapter $15$]{DaPZab1996ergo}). However, it seems that the Sobolevski\u{\i}-Kato-Fujita technique has not been applied before to the stochastic case. The following crucial lemma provides the new \textit{a priori} estimate for $v_n$ using a method inspired by \cite{Soboleskii1959onnonstationary} (see the introduction to \cite{kielhofer1980global}).
 \begin{lemma}\label{LEM:priola}
     Let $\eps,G$ be as in hypothesis \eqref{HP:H1}. For all $T>0$, $\gamma\in[1/4,1/4+ \eps )$ and $x\in D(A^\g)$ there exists a constant $c_2=c_2\big(\eps,G,\gamma,\|A^\g x\|,T\big)>0$ such that
     \begin{equation*}
\sup_{n\in\mN}\sup_{t\in[0,T]}\|A^{\gamma}v_n(t)\|\leq c_2 \qquad \mP-a.s.
\end{equation*}
The estimate is uniform for $x$ in bounded sets of $D(A^\g)$.
 \end{lemma}
 \begin{proof}
  \textit{Step $1$.} We start by estimating the $D(A^{\frac18})$-norm of the integral term in equation \eqref{EQ:v_nmild}, thanks to Lemma \ref{LEM:bB} with the choices $\delta=5/8, \rho=1/4, \theta=1/8$. We obtain, for a constant $C>0$ depending only on $T$ and on the choices of $\eps,G$
\begin{align*}
    J:=\ &\bigg\|A^{\frac18}\int_0^te^{-(t-s)A}B_n\big(v_n(s)+z(s)\big)\d s \bigg\|\\
    =\ &\bigg\|\int_0^tA^{\frac68}e^{-(t-s)A}A^{-\frac58}B_n\big(v_n(s)+z(s)\big)\d s \bigg\|\\
    \leq \ &\int_0^t\Big(\frac{3}{4e}\Big)^{3/4}(t-s)^{-3/4}c_0\big\|A^{\frac14}\big(v_n(s)+z(s)\big)\big\| \big\|A^{\frac18}\big(v_n(s)+z(s)\big)\big\|\d s\\
    \leq \ &C\int_0^t(t-s)^{-3/4}\Big(\|A^{\frac14}v_n(s)\|+1\Big)\Big(\|A^{\frac18}v_n(s)\|+1\Big)\d s ,
\end{align*}
where we used equation \eqref{EQ:Aaet1} and controlled uniformly in time $z$ thanks to Theorem \ref{TH:z}.
We now use the H\"older inequality with exponent $p>4$ and the respective $q=(1-1/p)^{-1}<4/3$. Constants $C>0$ hereafter may vary from line to line, yet they depend only on $\eps,G,T$.
 \begin{align*}
    J\leq \, &C\left[\int_0^t(t-s)^{-3q/4}\d s \right]^{1/q}\left[\int_0^t\Big(\|A^{\frac14}v_n(s)\|+1\Big)^p\Big(\|A^{\frac18}v_n(s)\|+1\Big)^p\d s \right]^{1/p}\\
    \leq \, &C\left[1+\|A^\eps x\|+\int_0^t\|A^{\frac18}v_n(s)\|^p\Big(\|A^{\frac14}v_n(s)\|^p+1\Big)\d s \right]^{1/p},
\end{align*}
where the first integral is finite because $-3q/4>-1$ and we controlled the $L^p(0,T;H)$-norm of $A^{\frac14}v_n$ thanks to Lemma \ref{LEM:atimaprioriferrario} (thus $p$ must be chosen such that $p<4/(1-2\eps)$). 
We now use this estimate into the mild formulation for $v_n$ (\textit{cf.} equation \eqref{EQ:v_nmild}) to control $\|A^{\frac18}v_n(t)\|$ $\mP-a.s.$ and for every $t\in[0,T]$: 
\begin{align*}
    \|A^{\frac18}v_n(t)\|^p
    &\leq 2^{p-1}\left(\|e^{-tA}\Pi_nA^{\frac18}x\|^p+J^p
    \right)\\
    &\leq C\bigg[1+\|A^{\frac14}x\|+\int_0^t\|A^{\frac18}v_n(s)\|^p
    \Big(\|A^{\frac14}v_n(s)\|^p+1\Big)\d s \bigg].
\end{align*}
The Gr\"onwall lemma applied $\mP-a.s.$ entails
\begin{align*}
\|A^{\frac18}v_n(t)\|^p&\leq C\big(1+\|A^{\frac14}x\|\big)\exp\bigg[c_0\int_0^{T}\|A^{\frac14}v_n(s)\|^p\d s\bigg].
\end{align*}
We can apply again Lemma \ref{LEM:atimaprioriferrario} and we obtain the uniform bound in $L^\infty\big(0,T;D(A^{\frac18})\big)$
\begin{equation}\label{EQ:estimate1/8}
\sup_{n\in\mN}\sup_{t\in[0,T]}\|A^{\frac18}v_n(t)\|\leq C\big(1+\|A^{\frac14}x\|\big)\exp\big[C\|A^{\frac14}x\|\big]\qquad \mP-a.s.
\end{equation}

 \textit{Step $2$.}  We now emulate the first step, but with the exponent $\gamma\in[1/4,1/4+ \eps )$ instead of $1/8$. We can apply Lemma \ref{LEM:bB} with the choices $\delta=7/8-\gamma, \rho=\gamma, \theta=1/8$ for every $t\in [0,T]$ and $\mP-a.s.$ to obtain
\begin{align*}
    \|A^{\gamma}v_n(t)\|&\leq \|e^{-tA}\Pi_nA^{\gamma}x\|+ \bigg\|A^{\gamma}\int_0^te^{-(t-s)A}B_n\big(v_n(s)+z(s)\big)\d s \bigg\|\\
    &\leq \|A^{\g} x\|+\bigg\|\int_0^tA^{\frac78}e^{-(t-s)A}A^{\gamma-\frac78}B_n\big(v_n(s)+z(s)\big)\d s \bigg\|\\
    &\leq \|A^{\g} x\|+ C\int_0^t(t-s)^{-7/8}\big\|A^{\gamma}\big(v_n(s)+z(s)\big)\big\| \ \big\|A^{\frac18}\big(v_n(s)+z(s)\big)\big\|\d s\\
    &\leq C\Big(1+\|A^{\g}x\|\exp\big[C\|A^{\g}x\|\big]\Big)\bigg[1+\int_0^t(t-s)^{-7/8} \|A^{\gamma}v_n(s)\|\d s\bigg],
\end{align*}
where in the last line we employed equation \eqref{EQ:estimate1/8} and Theorem \ref{TH:z}.
We resort now to the modified version of Gr\"onwall's lemma \ref{LEM:modifiedgronwall} to obtain
\begin{align*}
\sup_{t\in[0,T]}\|A^{\gamma}v_n(t)\|&\leq C\Big(1+\|A^{\g}x\|\exp\big[C\|A^{\g}x\|\big]\Big)\qquad \forall \, n\in\mN,\ \mP-a.s.
\end{align*}
 The assertion follows easily.
\end{proof}

\subsection{New path regularity}
We can use Lemma \ref{LEM:priola} to prove that the generalized solution to equation \eqref{EQ:NS} characterized by Theorem \ref{TH:knownresult} has trajectories with higher regularity in space. 

\begin{theorem}\label{TH:ulunardi}
    Let $\eps,G$ be as in hypothesis \eqref{HP:H1}. Let us take $\gamma\in[0,1/4+ \eps )$ and $x\in D(A^\g)$. Then the generalized solution $X^x$ to equation \eqref{EQ:NS} from Theorem \ref{TH:knownresult} has $\mP-a.s.$ paths in $C\big([0,T];D(A^\g)\big)$ for all $T>0$.
\end{theorem}
\begin{proof}\textit{Step $1.$} We take $\eps,G,\g,x,T$ as in the hypotheses and fix $\omega\in\Omega$ $\mP-a.s.$ We prove that a subsequence of $A^\g v_n$ (\textit{cf.} equation \eqref{EQ:v_nmild}) converges in $C\big([0,T];H\big)$. As for the term involving the initial datum, we directly have that $e^{\cdot A}x_n$ converges to $e^{-\cdot A}x$ in $C\big([0,T];D(A^\g)\big)$ as $n\to\infty$.

    As for the term with the non-linearity, we intend to apply the Arzelà-Ascoli theorem. 
    If we {rename $h_n:[0,T]\to H: t\mapsto \int_0^te^{-(t-s)A}B_n\big(v_n(s)+z(s)\big)\, \d s$},  we prove that $h_n$ is uniformly bounded in $C^{1/2-\g}\big([0,T];D(A^\g)\big)$, which gives equi-continuity.      In  \cite[Section $2.2.1$]{Lunardi95Analyticsemigroups}, the Banach space $D_A(\alpha,1)$ is introduced for any $\a\in(0,1)$\footnote{It consists of all $x\in H$ such that $s\mapsto s^{-\alpha}\|Ae^{-sA}x\|$ is in $L^1(0,1)$ and it is endowed with the complete norm $\|x\|_{D_A(\alpha,1)}:=\|x\|+\int_0^1s^{-\a}\|Ae^{-sA}x\|\d s$.}. The result \cite[Proposition $2.2.15$]{Lunardi95Analyticsemigroups} proves that $D_A(\alpha,1)$ is continuously embedded into $D(A^\a)$, which directly implies that for any $\b\in(0,1)$ the space of $\b$-H\"older continuous functions 
    $C^{\b}\big([0,T];D_A(\a,1)\big)$ is continuously embedded into $C^{\b}\big([0,T];D(A^\a)\big)$. We apply these results for $\alpha=1/2+\gamma\in[1/2,3/4+ \eps )\subset (0,1)$.
    \noindent Moreover, \cite[Proposition $4.2.1$]{Lunardi95Analyticsemigroups} gives 
    \[
    \|A^{-\frac12}h_n\|_{C^{1-\a}([0,T];D_A(\a,1))}\leq c\!\sup_{t\in[0,T]}\big\|A^{-\frac12}B_n\big(v_n(t)+z(t)\big)\big\|,
    \]
    for a constant $c>0$ that depends only on $\a$ and $T$. This considerations lead to the following estimates, where $C=C(\eps,\g,T)>0$ possibly varies from line to line
    \begin{align*}
    \|h_n\|_{C^{1/2-\g}([0,T];D(A^\g))}&=
    \|A^{-\frac12}h_n\|_{C^{1-\a}([0,T];D(A^\a))}\\
    &\leq C\|A^{-\frac12}h_n\|_{C^{1-\a}([0,T];D_A(\a,1))}\\
    &\leq C\sup_{t\in[0,T]}\big\|A^{-\frac12}B_n\big(v_n(t)+z(t)\big)\big\|\\
    &\leq C\Big[\|A^{\frac14}z\|^2_{C([0,T]:H)}+\sup_{t\in[0,T]}\|A^{\frac14}v_n(t)\|^2\Big]
    \end{align*}
    where we used almost surely Lemma \ref{LEM:bB} with $\delta=1/2$ and $\rho=\theta=1/4$. Finally Lemma \ref{LEM:priola} gives the uniform estimate in $n$.

    By arbitrariness of $\g$, we can consider $\g<\g'<1/4+ \eps $ and apply Lemma \ref{LEM:priola} with $\g'$ instead of $\g$. Since $D(A^{\g'})$ is compactly embedded into $D(A^\g)$, we have that $h_n(t)=e^{-tA}x_n-v_n(t)$, at any fixed $t\in[0,T]$, lies in a compact set of $D(A^\g)$.
    
    We thus apply the Arzelà-Ascoli theorem to the sequence $h_n$ and infer the existence of a sub-sequence  converging in $C\big([0,T];D(A^\g)\big)$ to a certain $h\in C\big([0,T];D(A^\g)\big)$. Let us define $\underline v:=e^{-\cdot A}x-h$, then a subsequence of $v_n$ converges to $\underline v$ in $C\big([0,T];D(A^\g)\big)$. We already discussed in Remark \ref{REM:v_ntov} that $v_n$ converges weakly* to $v$ in $L^\infty(0,T;H)$, thus $v=\underline v\in C\big([0,T];D(A^\g)\big)$.

    \textit{Step $2.$} 
We now define $u_n:=v_n+z$ for all $n\in\mN$ and $\mP-a.s.$ By the previous step we obtain uniform convergence in time for $u_n$ to the function $u:=v+z\in C\big([0,T];D(A^\g)\big)$. This function satisfies our definition of generalized solution (\textit{cf.} Definition \ref{DEF:gensol}). Indeed, by recalling the equation \eqref{EQ:defv} satisfied by $v$, we have for all $\phi\in D(A)$, for all times $t\geq 0$ and almost surely in probability
\begin{align*}\notag
     \langle u(t),\phi\rangle&=\langle v(t), \phi\rangle +\langle z(t), \phi\rangle\\&=-\int_0^t\langle v(s), A\phi\rangle\d s +\int_0^t b\big(v(s)+z(s), \phi, v(s)+z(s)\big) \d s + \langle x, \phi\rangle + \langle z(t), \phi\rangle\notag\\
     &=-\int_0^t\langle u(s), A\phi\rangle \d s + \int_0^t b\big(u(s), \phi, u(s)\big) \d s+\langle x, \phi\rangle +\langle z(t), \phi\rangle +\int_0^t\langle z(s), A\phi\d s \rangle\notag\\
     &=-\int_0^t\langle u(s), A\phi\rangle \d s + \int_0^t b\big(u(s), \phi, u(s)\big) \d s+\langle x, \phi\rangle +\langle GW_t,\phi \rangle,
 \end{align*}
where in the last line we employed Theorem \ref{TH:z}.    By the uniqueness result expressed in Theorem \ref{TH:knownresult}, it follows that the function $u$ we constructed by finite dimensional approximations coincides with the one in Theorem \ref{TH:knownresult}.
\end{proof}

\begin{remark}\label{REM:zdet}
It is worth noting that we only used so far the regularities for the trajectories $z$ of the stochastic convolution (\textit{cf.} Theorem \ref{TH:z}), reasoning at $\omega\in\Omega$ fixed almost surely. Therefore the theses in Theorem \ref{TH:ulunardi} still hold if we replace the stochastic convolution with a generic continuous deterministic function. More in detail: given $\gamma\in[0,1/2)$, $x\in D(A^\g)$ and  $z\in C\big([0,T];D(A^\g)\big)$ such that $z(0)=0$, the unique solution $v$ to equation \eqref{EQ:defv} has regularity $C\big([0,T];D(A^\g)\big)$.
\end{remark}

\begin{proposition}\label{PROP:t>0}
    Let $\eps,G$ be as in hypothesis \eqref{HP:H1}. Given $x\in H$, the generalized solution to equation \eqref{EQ:NS} from Theorem \ref{TH:knownresult} has $\mP-a.s.$ trajectories in $C\big([t_0,T];D(A^\g)\big)$ for all $0<t_0<T$ and all $\g\in[0,1/4+\eps)$.
\end{proposition}
\begin{proof}
    Let us take $x\in H$ and $\eps,G$ as in the hypotheses, then we know by Theorem \ref{TH:knownresult} that there exists a pathwise unique stochastic process $X^x$ with almost surely trajectories $u\in L^{2}\big(0,T;D(A^{\frac14+\eps})\big)$ for all $T>0$. We take $\g\in[0,1/4+\eps)$, thus $u(t)\in D(A^\g)$ for almost every $t>0$ and $\mP-a.s$. This allows us, once fixed $0<t_0<T$, to choose $t_1\in (0,t_0)$ such that $u(t_1)\in D(A^\g)$ $\mP-a.s.$, which can be chosen as a more regular starting point for equation \eqref{EQ:NS}. 
    Its unique generalized solution was proved to have paths in $C\big([0,T];D(A^\g)\big)$ almost surely: let us fix one of these trajectories, denoted by $w$, and its respective $\omega\in\Omega$.
We now define for that $\omega\in\Omega$
\begin{equation*}
    \tilde u(t)=
    \begin{cases}
        u(t)\qquad &\text{ if } t\in[0,t_1)\\
        w(t-t_1) &\text{ if } t\geq t_1
    \end{cases},
\end{equation*}
then $\tilde u$  satisfies the equation and the regularities in Definition \ref{DEF:gensol} for that $\omega\in\Omega$ and with starting point $x\in H$, therefore it must coincide with the trajectory $u$ of $X^x$ by the uniqueness result in Theorem \ref{TH:knownresult}. In particular we deduce that $t\mapsto u(t)=\tilde u(t)=w(t-t_1)$ is continuous and $D(A^\g)$-valued at every $t\geq t_1$, thus also $u\in C\big([t_0,T];D(A^\g)\big)$. We conclude by arbitrariness of $\omega\in\Omega$ $\mP-a.s.$
\end{proof}

\section{Application to invariant measure}\label{SEC:misura}
In this section we use the regularity result obtained in Section \ref{SEC:Mainresult} to prove the uniqueness  and the related ergodic properties of the invariant measure $\mu\in\mathcal P(H)$ provided by Theorem \ref{TH:munoto}. To this purpose we will also prove the strong Feller and irreducibility properties for the Markov semigroup $\{P_t\}_{t\geq 0}$ (see Section \ref{SEC:P_tmu}). We will always assume that the stochastic noise in equation \eqref{EQ:NS} is given by a cylindrical Wiener process $W$ in $H$ regularized by a linear operator $G$ that satisfies hypothesis \eqref{HP:H1}.
We adapt the reasoning from \cite{FlandoliMaslowski1995ergodicity,ferrario1997ergodic}, where however the uniqueness of the invariant measure is proved only under the stronger hypothesis \eqref{HP:H2}. 

We start by recalling two main properties for the Markov semigroup associated to equation \eqref{EQ:NS}, the first of which is classical and the second was introduced in \cite{FlandoliMaslowski1995ergodicity}. 
\begin{definition}\label{DEF:irredSF}
The Markov semigroup $\{P_t{\}}_{t\geq 0}$ introduced in Section \ref{SEC:P_tmu} 
\begin{itemize}
\item is irreducible on $D(A^\a)$, for some $\a\geq 0$, if for every time $t>0$, every point $x\in  D(A^\a) $ and all open non-empty sets $U\subset  D(A^\a) $, it holds $P_t^*\delta_x(U)=\mP(X^x_t\in U)>0$;
\item enjoys the \textit{(SF)} property on $ D(A^\a) \hookrightarrow H$, for some $\a>0$, if for every time $t>0$, every  function $\varphi\in\mathcal B_b\big( D(A^\a) \big)$ and every $x\in D(A^\a)$ we have
\[
\lim_{n\to\infty}P_t\varphi(x_n)=P_t\varphi(x)
\]
for every sequence $\{x_n\}_{n\in\mN}\subset D(A^\a)$ bounded in $D(A^\a)$ and converging to $x\in D(A^\a)$ with respect to the norm of $H$. 
\end{itemize}
\end{definition}
\begin{remark} \label{REM:irredSF}
    \begin{enumerate}[wide, label=$(\roman*)$]
    \item \label{IT:sigma} First of all, observe that the $\sigma$-algebra on $D(A^\a)$, $\a>0$,  generated by the norm $\|\cdot\|^{}_{D(A^\a)}=\|A^\a\cdot\|$ and denoted by $\mathscr{B}_{D(A^\a)}$ coincides with the one induced from $H$ and denoted by $\mathscr{B}_H \cap D(A^\a)$. 
    In particular, we have that all Borel subsets of $D(A^\a)$ are Borel subsets of $H$. A proof of this statement can be found in appendix (\textit{cf.} Lemma \ref{LEM:sigmainduced}).
    \item \label{IT:ISFcompare} A classical approach to proving uniqueness of the invariant measure for a Markov semigroup involves showing that the semigroup is both irreducible and strong Feller (\textit{cf.} \cite[Section $11.2$]{da2014stochastic}). This methodology was applied in \cite{ferrario1997ergodic} to our equation, with the stronger hypothesis \eqref{HP:H2}, to establish uniqueness of $\mu$ within the set of probability measures over $ \mathscr B_{D(A^{1/4})}$.\\    
    \noindent On the other hand, in \cite{FlandoliMaslowski1995ergodicity} the authors introduced two modified versions of irreducibility and strong Feller property, denoted respectively by \textit{(I)} and \textit{(SF)}, which granted uniqueness of the invariant measure in $\mathcal P(H)$ under the hypothesis \eqref{HP:H2}.\\    
    \noindent We have blended the two notions in Definition \ref{DEF:irredSF}. We will prove that they are sufficient, under the more general assumption \eqref{HP:H1} and thanks to the results in Section \ref{SEC:Mainresult}, to establish the uniqueness of the invariant measure within probabilities on $H$ and its concentration on a suitable Borel subset. 
    \item It is straightforward that the $\textit{(SF)}$ property on $D(A^\a)\hookrightarrow H$ is stronger then the usual strong Feller property on $D(A^\a)$, which would prevent us from obtaining uniqueness of the invariant measure over all probabilities on $H$.
    It is worth mentioning that the $\textit{(SF)}$ property on $D(A^\a)\hookrightarrow H$ is however weaker then requiring the continuity of $P_t\varphi$ with respect to the norm of $H$ for all $t>0$ and $\varphi\in \mathcal B_b\big(D(A^\a)\big)$. This would actually be sufficient for our goals, but we were not able to prove it.
    \end{enumerate}
\end{remark}

We first present a lemma, adapted from \cite[Theorem $4.1$]{FlandoliMaslowski1995ergodicity}, which is auxiliary for our main result. This lemma does not depend on the fact that the Markov semigroup generates from our equation \eqref{EQ:NS}, thus we state it in a general form as follows. The proof can be found in appendix.

\begin{lemma}\label{LEM:aux}
If a Markov semigroup $\{P_t\}_{t\geq 0}^{}$ on $H$ is both irreducible in $ D(A^\a) $ and enjoys the \textit{(SF)} property on $ D(A^\a) \hookrightarrow H$, for some $\a>0$, then the transition probabilities $P_t^*\delta_x$ for $t>0$ and $x\in  D(A^\a) $ are mutually equivalent on $\mathscr B_{ D(A^\a) }$.
\end{lemma}

We now state the main theorem of the section, which improves Theorem \ref{TH:munoto}.

\begin{theorem}\label{TH:ergo}
    Let $P=\{P_t\}_{t\geq 0}$ denote the Markov semigroup in $H$ associated to the solution of equation \eqref{EQ:NS} (\textit{cf.} Section \ref{SEC:P_tmu}) under the hypothesis \eqref{HP:H1}. Let $\mu$ be the invariant measure of Theorem \ref{TH:munoto}. If $P$ is both irreducible on $ D(A^{\frac14})$ and enjoys the \textit{(SF)} property on $ D(A^{\frac14})\hookrightarrow H$, then 
    \begin{enumerate}[label=$(\roman*)$]
        \item\label{IT:uni} $\mu$ is unique in the set $\mathcal P(H)$ and is concentrated on $ D(A^{\g})$ for all $\g<1/4+ \eps $,
        \item\label{IT:erg} $\mu$ is ergodic, \textit{i.e.} for all $\varphi\in\mathcal B_b(H)$, it holds
        \[
        \ds\lim_{T\to\infty}\frac1T\int_0^TP_t\varphi\d t = \int_H\varphi\d \mu,
        \]
        \item\label{IT:equi} $\mu$ is equivalent to $P_t^*\delta_x$ for all $t>0$ and $x\in H$,
        \item\label{IT:point} for all $U\in\mathscr B_H$ and $x\in H$ it holds $\ds\lim_{t\to+\infty}P_t^*\delta_x(U)=\mu(U)$,
        \item\label{IT:TV} $\mu$ is strongly mixing, \textit{i.e.} for { every $\nu\in\mathcal P(H)$} it holds $\ds \lim_{t\to+\infty}\|P_t^*\nu-\mu\|_{TV}=0$.
    \end{enumerate}
\end{theorem}
\begin{proof}
\textit{Step $1.$} We first apply the auxiliary Lemma \ref{LEM:aux} with the choice $\a=1/4$ to obtain that the transition probabilities $P_t^*\delta_x$ for $t>0$ and $x\in  D(A^{\frac14})$ are mutually equivalent on 
$\mathscr B_{ D(A^{{1/4}}) }$.

\textit{Step $2.$} We now prove the equivalence on $\mathscr B_{ H^{}}$ of the transition probabilities $P_t^*\delta_x$ for $t>0$ and $x\in  H^{}$. This reasoning is inspired by \cite[Lemma $4.1$]{FlandoliMaslowski1995ergodicity}, which however required the stronger assumption \eqref{HP:H2}.

    Let $t,s>0$, $x_0,x\in H$ and $U\in\mathscr B_H$ and let us suppose that $P_t^*\delta_{x_0}(U)>0$, then the thesis is proved if we show that $P_s^*\delta_x(U)>0$. We use the Chapman-Kolmogorov equation with intermediate time $h\in (0,s\wedge t )$:
    \begin{align}\label{EQ:lemmamutuallyequivalent}
        0<P_t^*\delta_{x_0}(U)=&\int_HP_{t-h}\mi_U(y)\, P_h^*\delta_{x_0}(\!\d y)\notag\\
        =&\int_{ D(A^{1/4})} P_{t-h}\mi_U(y) \,P_h^*\delta_{x_0}(\!\d y)+\int_{H\setminus  D(A^{1/4})}P_{t-h}\mi_U(y) \,P_h^*\delta_{x_0}(\!\d y)\\
        =&\int_{ D(A^{1/4})}P_{t-h}\mi_U(y) \,P_h^*\delta_{x_0}(\!\d y).\notag
    \end{align}
    The last equality holds because we know by Proposition \ref{PROP:t>0} that $X^{x_0}_h\in  D(A^{\frac14})$ $\mP$-a.s., thus
    \[
    0\leq \int_{H\setminus  D(A^{1/4})}P_{t-h}\mi_U(y)\, P_h^*\delta_{x_0}(\!\d y )\leq P_h^*\delta_{x_0}\big(H\setminus  D(A^{\frac14})\big)= \mP\big(X^{x_0}_h\in H\setminus  D(A^{\frac14})\big)=0.
    \]
    We deduce by the strict inequality in equation \eqref{EQ:lemmamutuallyequivalent} that there exists $y_0\in  D(A^{\frac14})$ such that 
    \begin{align*}
     0<P_{t-h}\mi_U(y_0)=\mP\big(X_{t-h}^{y_0}\in U\cap  D(A^{\frac14})\big)+\mP\big(X_{t-h}^{y_0}\in U\setminus D(A^{\frac14})\big)
     & =P_{t-h}^*\delta_{y_0}\big(U\cap  D(A^{\frac14})\big),
    \end{align*}
    where we used that $X^{y_0}_{t-h}\in  D(A^{\frac14})$ almost surely.
    We now resort to the hypothesis on the equivalence of transition probabilities over $\mathscr B_{D(A^{1/4})}=\mathscr B_H\cap  D(A^{\frac14})$ to infer that 
    \begin{equation}\label{EQ:lemmaeq2}
    P_{s-h}^*\delta_y\big(U\cap D(A^{\frac14})\big)>0\qquad \forall \,y\in  D(A^{\frac14}).
    \end{equation}
    Therefore, by the fact that $X^x_s, X^x_h\in  D(A^{\frac14})$ almost surely and the Chapman-Kolmogorov equation, we conclude that 
    \begin{align*}
        P_s^*\delta_x(U)&=\mP\big(X_{s}^{x}\in U\cap  D(A^{\frac14})\big)+
        \mP\big(X_{s}^{x}\in U\setminus  D(A^{\frac14})\big)\\
         &=P_s^*\delta_x\big(U\cap  D(A^{\frac14})\big)\\
         &=\int_H P_{s-h}\mi_{U\cap  D(A^{1/4})}(y)\,P_h^*\delta_x(\!\d y )\\
         &=\int_{ D(A^{1/4})} P_{s-h}\mi_{U\cap  D(A^{1/4})}(y)\, P_h^*\delta_x(\!\d y)+\int_{H\setminus  D(A^{1/4})} P_{s-h}\mi_{U\cap  D(A^{1/4})}(y)\, P_h^*\delta_x(\!\d y)\\
         &=\int_{ D(A^{1/4})} P_{s-h}\mi_{U\cap  D(A^{1/4})}(y)\, P_h^*\delta_x(\!\d y)>0.
    \end{align*}
    If indeed the last strict inequality were false than we we would have $0=P_{s-h}\mi_{U\cap  D(A^{1/4})}(y)$ for almost every $y\in  D(A^{\frac14})$, which contradicts equation \eqref{EQ:lemmaeq2}.

    \textit{Step $3.$} Thanks to  the previous step, Doob's Theorem implies \itemref{IT:erg}, \itemref{IT:equi}, \itemref{IT:point} and the uniqueness part in \itemref{IT:uni} (\textit{cf.} \cite[Theorem $11.14$]{da2014stochastic}). 
    Concentration of $\mu$ in $ D(A^{\g})$ for $\g>0$ close to $1/4+ \eps $ follows by \itemref{IT:equi}, thanks to Proposition \ref{PROP:t>0}. Part \itemref{IT:TV} is due to Seidler (\textit{cf.} \cite[Proposition $2.5$]{seidler1997ergodic}).
\end{proof}

In what follows, it is proved that the Markov semigroup associated with the solution of equation \eqref{EQ:NS} (see Section \ref{SEC:P_tmu}) does indeed satisfy the hypotheses of Theorem \ref{TH:ergo}, thus reaching all the uniqueness and ergodic properties there listed.

\subsection{Irreducibility}
If $E$ is a normed vector space, we denote with $C_0\big([0,T];E\big)$ the closed subspace of $C\big([0,T];E\big)$, with the induced norm, of all functions $z$ such that $z(0)=0$. 
For every $x\in H$ and continuous $z:[0,+\infty)\to H$ with $z(0)=0$, we study the deterministic abstract equation in $H$ given in system \eqref{EQ:systemv}. Its solution is intended in the generalized sense, namely as a function $v\in C\big([0,T];H\big)\cap L^2(0,T;V)$ for all $T>0$ that satisfies equation \eqref{EQ:defv} in the deterministic frame.

We adapt the proof of \cite[Lemma $5.3$]{FlandoliMaslowski1995ergodicity} thanks to the regularity we obtained in Theorem \ref{TH:mainresult}.
\begin{lemma}\label{LEM:Phi}
For any $\g\in(1/4,1/2)$, $x\in D(A^\g)$ and every $T>0$ the following map $\Phi=\Phi^{x,T,\g}$ is well-defined
    \begin{align*}
    \Phi:C_0\big([0,T];D(A^\g)\big)\to C\big([0,T]; D(A^{\frac14})\big):    z\mapsto v+z,
    \end{align*}
    where $v$ is the unique solution in $C\big([0,T];H\big)\cap L^2(0,T;V)$ to system \eqref{EQ:systemv} in the generalized sense expressed by equation \eqref{EQ:defv}. Moreover:
    \begin{enumerate}[label=(\alph*)]
    \item\label{IT:Phicontinuous}  The map $\Phi$ is continuous with respect to the assigned topologies.
    \item\label{IT:datixyesistebarz} For every $y\in D(A^\g)$  there exists $ \bar z\in C_0\big([0,T];D(A^\g)\big)$ such that $\Phi( \bar z)(T)=y$.
    \end{enumerate}
\end{lemma}
\begin{proof} For fixed $\g\in(1/4,1/2)$, $x\in D(A^\g)$ and $T>0$, the well-posedness of $\Phi$ follows from Remarks \ref{REM:v_ntov} and \ref{REM:zdet}. 
As far as part \ref{IT:Phicontinuous} is concerned, we take an arbitrary $z_0\in C_0\big([0,T];D(A^\g)\big)$ and prove continuity of $\Phi$ at $z_0$. We choose $z$ close to $z_0$
and we show that $\Phi(z)$ is close to $\Phi(z_0)$, in the respective norms.

We denote $u_0=\Phi(z_0), v_0=\Phi(z_0)-z_0$ and $u=\Phi(z), v=\Phi(z)-z$. We know by Remark \ref{REM:zdet} that $u_0,v_0,u,v\in C\big([0,T];D(A^\g)\big)$, which guarantees that $B(u_0),B(u)\in C\big([0,T];V'\big)$ by Lemma \ref{LEM:bB}. We introduce for brevity $Z=z_0- z\in C_0\big([0,T];D(A^\g)\big)$, $Y=v_0- v\in L^2(0,T;V)\cap C\big([0,T];D(A^\g)\big)$ and $U=Y+Z=u_0-u\in C\big([0,T];D(A^\g)\big)$. Then $Y$ satisfies for all $\phi\in D(A)$ and $t\geq 0$
\begin{equation*}
\begin{split}
    \langle Y(t), \phi\rangle +\int_0^t\langle Y(s), A\phi\rangle \d s = \int_0^t b\big(u_0(s), \phi,u_0(s)\big) \d s-\int_0^t b\big(u(s), \phi,u(s)\big) \d s,
    \end{split}
\end{equation*}
which implies, by the regularity of the terms involved
\begin{equation}\label{EQ:Vphi}
\begin{split}
    \prescript{}{ V'}{\langle} Y(t), \phi{\rangle}_V +\int_0^t\prescript{}{ V'}{\langle} AY(s), \phi{\rangle}_V \d s = \int_0^t \prescript{}{ V'}{\big\langle}B(u(s)),\phi{\big\rangle}_V\d s-\int_0^t \prescript{}{ V'}{\big\langle}B(u_0(s)),\phi{\big\rangle}_V\d s.
    \end{split}
\end{equation}
By density of $D(A)$ in $V$ and continuity of the duality pairing, equation \eqref{EQ:Vphi} can be written for all $\phi\in V$, which implies by a standard result in functional analysis (\textit{cf.} \cite[Chapter $3$, Lemma $1.1$]{temam2001navier}) that $Y\in H^1(0,T;V')$ and 
\[
Y'(t)+AY(t)=B(u(t))-B(u_0(t))\qquad a.e.\, t>0,
\]
the equality holding in $V'$. We take $a.e.$ in time the pairing in $V'/V$ with $Y$ to obtain
\begin{equation}  \label{EQ:Vphi2}
    \prescript{}{ V'}{\langle} Y'(t), Y(t){\rangle}_{ V}+
    \prescript{}{ V'}{\langle} AY(t), Y(t){\rangle}_{ V}=
    \prescript{}{ V'}{\big\langle} B(u(t))-B(u_0(t)), Y(t){\big\rangle}_{ V}\qquad a.e.\  t>0.
\end{equation}
We recall that (see \cite[Chapter $3$, Lemma $1.2$]{temam2001navier})
\[
\int_0^t\prescript{}{ V'}{\langle} Y'(s), Y(s){\rangle}_{ V}\d s =\frac12\|Y(t)\|^2\qquad \forall \, t\geq 0
\]
and that $\prescript{}{ V'}{\langle} AY(t), Y(t){\rangle}_{ V}=\|Y(t)\|^2_V$ for almost every $t>0$.
Moreover almost everywhere in time it holds (see Lemma \ref{LEM:bB})
\begin{equation*}
    \prescript{}{ V'}{\big\langle} B(u)-B(u_0), Y{\big\rangle}_{ V}= b(u_0,Y,Y) 
    - b(u_0,Z,Y)
    -b(U,u,Y).
\end{equation*}
The three terms in the right-hand side can be controlled as follows for some $\epsilon>0$, thanks to the uniform bounds in time on $u,u_0$
\begin{align*}
\big|b(u_0,Y,Y)\big|
&\leq c_0\|A^{\frac14}u_0\|\|A^{\frac14}Y\|\|Y\|_{V}^{}
\leq C\|Y\|^{1/2}\|Y\|^{1/2}_V\|Y\|_V^{}
\leq \epsilon\|Y\|_V^2+C_\epsilon\|Y\|^2\\
\big|b(u_0,Z,Y)\big| 
&\leq c_0\|A^{\frac14}u_0\|\|A^{\frac14}Z\|\|Y\|_V^{}
\leq \epsilon\|Y\|_{V}^{2}+C_\epsilon\|A^{\g}Z\|^2\\
\big|b(U,u,Y)\big| 
&\leq c_0\|A^{\frac14}(Y+Z)\|\|A^{\frac14}u\|\|Y\|_V^{}
\leq C\big(\|Y\|_{V}^{1/2}\|Y\|^{1/2}_V+\|A^{\frac14}Z\|\big)\|Y\|_{V}^{}\\
&\leq \epsilon\|Y\|_{V}^{2}+C_\epsilon||Y||^2+C_\epsilon\|A^{\g}Z\|^2.
\end{align*}

Therefore if we integrate equation \eqref{EQ:Vphi2} on the time interval $[0,t]$, we obtain for all times
\begin{gather*}
\frac12\|Y(t)\|^2+(1-3\epsilon)\int_0^t\|Y(s)\|^2_V\d s  
\leq C_\epsilon
\bigg[\|Z\|^2_{C([0,T];D(A^\g))}+\int_0^t\|Y(s)\|^2\d s\bigg].
\end{gather*}
After choosing in the last equation $1-3\epsilon>0$, Gr\"onwall's lemma entails
\begin{align*}
\|Y(t)\| \leq  C\|Z\|_{C([0,T];D(A^\g))}.
\end{align*}
By means of the interpolation inequality (\textit{cf.} Lemma \ref{LEM:sobolevinterpolation}) one obtains, for $\lambda:=1-\frac{1}{4\g}$
\[
\|A^{\frac14}Y(t)\|\leq \|Y(t)\|^\la\, \|A^{\g}Y(t)\|^{1-\la}\leq C\|Z\|^\la_{C([0,T];D(A^\g))}.
\]
We finally proved continuity at $z_0$ for $\Phi$:
\begin{gather*}
\|\Phi(z)-\Phi( z_0)\|_{C([0,T]; D(A^{1/4}))}=\sup_{t\in[0,T]}\|A^{\frac14}Y(t)+A^{\frac14}Z(t)\|\\
\leq C\Big(\|z- z_0\|^\lambda_{C([0,T];D(A^\g))}+\|z- z_0\|^{}_{C([0,T];D(A^\g))}\Big).
\end{gather*}

    As far as part \itemref{IT:datixyesistebarz} is concerned, we follow the approach by \cite{ferrario1997ergodic}. For fixed $T>0$, starting point $x\in D(A^\g)$ and terminal point $y\in D(A^\g)$ we construct the following function $\bar u$, by choosing $0<t_0<t_1<T$
    \begin{equation*}
    \bar u (t)=
        \begin{cases}
            e^{-tA}x &\text{ if }t\in [0,t_0]\\
            e^{-(T-t)A}y&\text{ if }t\in [t_1,T]\\
            \ds \bar u(t_0)+\frac{t-t_0}{t_1-t_0}\big(\bar u (t_1)-\bar u (t_0)\big) \qquad &\text{ if }t\in (t_0,t_1)
        \end{cases}.
    \end{equation*}
    By direct inspection, thanks to equation \eqref{EQ:Aaet1} (with $\alpha=1/2-\gamma$) and the bound $\gamma>1/4$,  we have $\bar u\in C\big([0,T];D(A^\gamma)\big)\cap L^4(0,T;D(A^{\frac12})\big)$. This in turn implies, by Lemma \ref{LEM:bB} (choosing $\delta=1/2-\gamma$ and $\rho=\theta=1/2$) and the bound $\gamma<1/2$, that $B(\bar u)\in L^2\big(0,T;D(A^{\frac12-\gamma})'\big)$.
    Define now $\bar v$ as the weak solution of the following linear differential equation, set in $D(A^\g)$
    \begin{equation*}
        \begin{cases}
            \bar v'+A\bar v=-B(\bar u)\\
            \bar v(0)=x
        \end{cases}.
    \end{equation*}
    We know by a standard result in linear parabolic PDEs (\textit{cf.} \cite[Section $11.1.2$]{renardy2004introduction}) that there exists a unique $\bar v\in H^1\big([0,T];D(A^{\frac12-\gamma})'\big)\cap L^2\big(0,T;D(A^{\frac12+\g})\big)\subset C\big([0,T];D(A^\g)\big)$ such that $\bar v(0)=x$ and 
    \begin{equation*} 
    {\llangle} \bar v'(t), \phi{\rrangle}+
    {\llangle} A\bar v (t),\phi{\rrangle}=
    -{\big\langle\!\big\langle} B(\bar u (t)), \phi{\big\rangle\!\big\rangle}\qquad \forall \,\phi\in  D(A^{\frac12-\g}) , \ a.e.\  t>0,
\end{equation*}
where we denote by $\llangle\,\cdot\,,\,\cdot\,\rrangle$ the duality pairing between $D(A^{\frac12-\g})'$ and $D(A^{\frac12-\g})$.
By limiting test functions to $\phi\in D(A)\subset  D(A^{\frac12-\g}) $ and integrating over $[0,t]\subset[0,T]$, we reach
\begin{align}\label{EQ:vbar}
    \langle \bar v(t), \phi\rangle +\int_0^t\langle \bar v(s), A\phi\rangle \d s 
    = \int_0^t b\big(\bar u(s),\phi, \bar u(s)\big) \d s+\langle x, \phi\rangle\qquad \forall \, \phi\in D(A), \ \forall \, t\in[0,T].
\end{align}
Then we conclude that $\bar z:=\bar u -\bar v$ satisfies the thesis of the theorem. Indeed equation \eqref{EQ:vbar} tells us that $\bar v$ satisfies in the generalized sense the equation
    \begin{equation*}
        \begin{cases}
            v'+Av+B(v+\bar z)=0\\
            v(0)=x
        \end{cases},
    \end{equation*} 
    which is known to have a unique generalized solution as discussed at the beginning of the proof. Therefore we have $\Phi(\bar z)=\bar z+\bar v=\bar u$, which entails the thesis.
\end{proof}

\begin{lemma}\label{TH:zmisurapiena}
    Let $\eps,G$ be as in hypothesis \eqref{HP:H1}. For all $\g\in[0,1/4+ \eps )$, the law of $Z$ as a random variable\footnote{The stochastic process $\{Z_t\}_{t\in[0,T]}$ can be seen as a random variable $\Omega\to C_0\big([0,T];D(A^\g)\big)$ thanks to \cite[Proposition $3.18$]{da2014stochastic} and our Theorem \ref{TH:z}.} in $C_0\big([0,T];D(A^\g)\big)$ is full, \textit{i.e.} $\mP(Z\in U)>0$ for every open non-empty subset $U$ of $C_0\big([0,T];D(A^\g)\big)$.
\end{lemma}
\begin{proof}
    This lemma is a direct application of \cite[Lemma $2.6$]{maslowski1993probability} once checked that its hypotheses (denoted by $(A1)$ and $(X2)$) are satisfied in our case. Adopting the notations used in the cited reference, we choose $D(A^\g)$ as the space $E$, so that $\hat E=C_0\big([0,T];D(A^\g)\big)$. Moreover $Q^0$ is the law of the random variable $Z:\Omega\to C_0\big([0,T];D(A^\g)\big)$ and condition $(A1)$ (about $Z$ having a continuous version in $E$) is satisfied thanks to our Theorem \ref{TH:z}.  Eventually the technical condition $(X2)$ is met thanks to \cite[Proposition $2.7$]{maslowski1993probability}, where $Q^{1/2}=G$ is a linear bounded and invertible operator from $H$ to $\Ran(G)$, which is densely and continuously embedded into $D(A^{\g})$.
    The hypotheses of \cite[Lemma $2.6$]{maslowski1993probability} are thus satisfied and the thesis states that the closure of $\operatorname{supp}(Q^0)$ in the topology of $\hat E$ coincides with $\hat E$. 
    Let now $U$ be an open non-empty subset of $C_0\big([0,T]; D(A^{\g})\big)=\hat E$, then the intersection between $U$ and $S:=\operatorname{supp}(Q^0)$ is non-empty\footnote{If it were empty, we would reach a contradiction: 
    $U\cap S=\varnothing \Longleftrightarrow S\subset U^c
        \Longrightarrow\hat E=\overline{ S}\subset \overline {U^c}=U^c
        \Longleftrightarrow U\subset {\hat E}^c=\varnothing.$}.
   Eventually by definition of support of a probability distribution\footnote{The support of a probability distribution $Q^0$ on a topological space is the largest Borel set $S$ with the following property: if $U$ is an open set with non-empty intersection with $S$, then $Q^0(U\cap S)>0$.} we have     $     \mP(Z\in U)=Q^0(U)\geq Q^0(U\cap {S})>0$.
\end{proof}
\begin{proposition}
    Let $\eps,G$ be as in hypothesis \eqref{HP:H1}. The Markov semigroup $\{P_t\}^{}_{t\geq 0}$ (\textit{cf.} Section \ref{SEC:P_tmu}) is irreducible in $ D(A^{\frac14})$.
\end{proposition}

\begin{proof}
 Let us fix $T>0$, $x\in  D(A^{\frac14})$ and $U$ open non-empty set in $ D(A^{\frac14})$, we prove that $P_T^*\delta_x(U)>0$. Fix $\eps$ and $G$ as in the hypothesis and take $\g\in(1/4,1/4+ \eps )$. By density of $D(A^\g)$ in $ D(A^{\frac14})$, we find $y\in D(A^\g)$ and $\delta>0$ such that the open ball $B^{ D(A^{1/4})}_\delta(y):=\{w\in  D(A^{\frac14}) : \|w-y\|^{}_{ D(A^{1/4})}<\delta\}$ is included in $U$. For that $y\in D(A^\g)$, part \itemref{IT:datixyesistebarz} of Lemma \ref{LEM:Phi} returns a $\bar z\in C_0\big([0,T];D(A^\g)\big)$ such that $\Phi(\bar z)(T)=y$ and $\mP-a.s.$
 \begin{align*}
 \|u(T)-y\|^{}_{ D(A^{1/4})}=\|\Phi(Z)(T)-\Phi(\bar z)(T)\|^{}_{ D(A^{1/4})}
 \leq \|\Phi(Z)-\Phi(\bar z)\|_{C([0,T]; D(A^{1/4}))}.
 \end{align*}
For the above $\delta>0$, continuity of $\Phi$ at $\bar z$ (see part \itemref{IT:Phicontinuous} of Lemma \ref{LEM:Phi}) returns $\rho>0$ such that, whenever $\|Z-\bar z\|^{}_{C([0,T];D(A^\g))}<\rho $, then $\|\Phi(Z)-\Phi(\bar z)\|_{C([0,T]; D(A^{1/4}))}<\delta$. We conclude by monotonicity of probability and Lemma \ref{TH:zmisurapiena}
\[
P_T^*\delta_x(U)\geq \mP\big(\|u(T)-y\|^{}_{ D(A^{1/4})}<\delta\big)\geq \mP\big(\|Z-\bar z\|^{}_{C([0,T];D(A^\g))}<\rho \big)>0.
\]
\end{proof}
\begin{remark}
Irreducibility in $ D(A^{\frac14})$ implies irreducibility in $H$ for our semigroup. Let indeed $U$ be an open set in $H$, then $U\cap  D(A^{\frac14})$ is open in $ D(A^{\frac14})$, as discussed in Remark \ref{REM:irredSF} \ref{IT:sigma}. Moreover, if $x\in H$, then we proved in Proposition \ref{PROP:t>0} that for all $t>0$ we have $X^x_{t}\in  D(A^{\frac14})$ almost surely. Therefore, given irreducibility in $ D(A^{\frac14})$, $x\in H$ and $t>0$, we find $t_0\in (0,t)$ such that $X^x_{t_0}=x_0\in  D(A^{\frac14})$ almost surely, thus  $$P_t^*\delta_x(U)\geq P_t^*\delta_x\big(U\cap  D(A^{\frac14})\big)=P_{t-t_0}^*\delta_{x_0}\big(U\cap  D(A^{\frac14})\big)>0.$$
\end{remark}
\subsection{\textit{(SF)} property}

As far as the \textit{(SF)} property is concerned, we proceed by finite dimensional approximations of a suitably modified version of our 2D N-S stochastic equation \eqref{EQ:NS}, similarly to \cite{FlandoliMaslowski1995ergodicity,ferrario1997ergodic}.

 For any $n\in\mN$ let $\Pi_n$ be the projector onto the subspace of $H$ generated by the eigenvectors  $e_k$ of $A$ for $k\in\{1,\dots,n\}$. 
 For every $R\in(0,+\infty)$ let $\Theta_R:[0,+\infty)\to[0,1]$ be a smooth function with compact support $[0,R+1]$ and equal to $1$ on the interval $[0,R]$.
 
 We study the following modified and approximated version of the stochastic 2D N-S equations, in which the function $B$ is smoothly truncated to $0$ whenever the norm in $ D(A^{\frac14})$ of the solution exceeds the parameter $R$. For every positive integer $n\in\mN$ we denote with $W^n:=\Pi_nW=\sum_{k=1}^n w^ke_k$ a finite dimensional standard Brownian motion in $H_n=\Pi_nH$, then for every real parameter $R>0$, every injective $G\in\mathcal L(H)$ and every $x\in H_n$ we consider 
 \begin{equation}\label{EQ:NSmodifiedapproximated}
     \begin{cases}
         \d X^{R,n}_t+AX^{R,n}_t\d t +\Theta_R\big(\|A^{\frac14}X^{R,n}_t\|^{2}\big)B_n\big(X^{R,n}_t\big)\d t = \Pi_nG\d W^n_t\qquad t>0,\, \mP-a.s.\\
         X^{R,n}_0=x\qquad \mP-a.s.
     \end{cases}
 \end{equation}
 This is an autonomous finite dimensional stochastic differential equation, with an additive noise and a drift which is the sum of a linear term and a globally Lipschitz non-linearity. Therefore, according to the standard existence and uniqueness theorem for solutions to SDEs under regular and Lipschitz coefficients (refer to, for example, \cite[Theorems $9.2$, $9.6$]{baldi2017stochastic}), there exists a pathwise unique Markov process $X^{R,n}$, which is a solution to the modified and approximated system \eqref{EQ:NSmodifiedapproximated}. When we want to highlight the dependence on the initial datum, we use the notation $X^{R,n}(x)$. Let $P^{R,n}$ denote the Markov semigroup associated with this solution. 
 It is known by classic theory (\textit{cf.} \cite[Chapter $1$]{cerrai2001second}) that the function $H_n\to H_n: x\mapsto X^{R,n}_s(x)$ is mean-square differentiable, therefore, by means of the Bismut-Elworthy formula (\textit{cf.} \cite{elworthy1994formulae, bismut1981martingales}), we know that $P_t^{R,n}\varphi$ is differentiable at all times $t> 0$ {for all $\varphi\in  C_b(H_n)$} and that 
 \begin{align}\label{EQ:Bismut}
  \nabla P^{R,n}_t\varphi(x)\boldsymbol\cdot h&=\frac{1}{t}\mE\left[\varphi\big(X^{R,n}_t(x)\big)\int_0^t(\Pi_nGG^*\Pi_n)^{-1/2}D_x\big[X^{R,n}_s(x)\big]h \boldsymbol\cdot  \d W^n_s \right]\qquad \forall \, x,h\in H_n.
 \end{align}
 We denoted with $D_x[X^{R,n}_s(x)]h$ the row-by-column multiplication between the Jacobian matrix evaluated at $x$ of the function $H_n\to H_n: x\mapsto X^{R,n}_s(x)$ and the vector $h\in H_n$. 

 \begin{lemma}\label{LEM:PRn}
     For any $G$ as in hypothesis \eqref{HP:H1} and for all  $R,t>0$ there exists a constant $C_R(t)>0$ such that, for all $n\in\mN$
  \begin{equation*}
\big|P_t^{R,n}\varphi(x)-P_t^{R,n}\varphi(y)|\leq C_R(t)\|\varphi\|_{\infty}\|x-y\|\qquad \forall \, \varphi\in C_b(H_n), \ \forall \,x,y\in H_n.
 \end{equation*}
 \end{lemma}
 \begin{proof}
Let us fix $R,t>0$, $n\in\mathbb N$ and take $\varphi:H_n\to\R$ bounded and continuous, then $P_t^{R,n}\varphi$ is differentiable at all times $t>0$ by the Bismut-Elworthy formula. 
By the mean value theorem we have, for all $x,y\in H_n$ and $t>0$
 \begin{gather}\label{EQ:lemmaBismut}\notag
  \big|P^{R,n}_t\varphi(x)-P^{R,n}_t\varphi(y)|\leq \sup_{z\in H_n}\big|\nabla P_t^{R,n}\varphi(z)\boldsymbol{\cdot}(x-y)\big|\\
  \leq \frac{1}{t}\|\varphi\|_{\infty}
\sup_{z\in H_n}\left(\mE\int_0^t\big\|(\Pi_nGG^*\Pi_n)^{-1/2} D_z\big[X^{R,n}_s(z)\big](x-y)\big\|^2\d s\right)^{1/2},
  \end{gather}
where we employed equation \eqref{EQ:Bismut} with $h=x-y$, controlled the first factor inside the expectation via the boundedness of $\varphi$, and It\^o's isometry.
Since $V\subset \Ran(G)$ and by adapting known estimates from \cite{FlandoliMaslowski1995ergodicity} (details are provided in the appendix, see Propositions \ref{PROP:casini} and  \ref{PROP:estderi}), we get 
\[
\mE\int_0^t\big\|(\Pi_nGG^*\Pi_n)^{-1/2} D_z\big[X^{R,n}_s(z)\big](x-y)\big\|^2\d s \leq 
\mE\int_0^t{\big\|D_z\big[X^{R,n}_s(z)\big](x-y)\big\|}_V^2\d s
\leq C_R(t)\|x-y\|^2,
\]
for a certain positive constant depending only on $R$ and $t$.
By inserting this result back into equation \eqref{EQ:lemmaBismut} we reach the thesis of the lemma for a possibly different constant $C_R(t)$.
 \end{proof}

  We study the following modified version of the stochastic N-S equations for fixed $R>0$ and $x\in H$
\begin{equation}\label{EQ:NSmodified}
    \begin{cases}
         \d X^R_t+AX^R_t\d t +\Theta_R\big(\|A^{\frac14}X^R_t\|^{2}         \big)B\big(X^R_t\big)\d t =G\d W_t\qquad t>0,\, \mP-a.s.\\
         X^R_0=x\qquad \mP-a.s.
     \end{cases}
\end{equation}
where the non linearity is continuously truncated to $0$ as soon as the $ D(A^{\frac14})$-norm of the solution exceeds the parameter $R$.

\begin{lemma}\label{LEM:allisfine}
    The results presented in Section \ref{SEC:Mainresult} still hold true if we replace equation \eqref{EQ:NS} with its modified version \eqref{EQ:NSmodified}. In particular for all $G$ as in hypothesis \eqref{HP:H1} and for all $x\in  D(A^{\frac14})$ we have $ X^{R}\in C\big([0,T]; D(A^{\frac14})\big)$ for all $R,T>0$ and $\mP-a.s.$
    Moreover a sub-sequence of $X^{R,n}$ converges to $X^R$ in $C\big([0,T]; D(A^{\frac14})\big)$ for all $T>0$, $\mP-a.s.$ and uniformly both in $R$ and in $x$, for $x$ in bounded sets of $ D(A^{\frac14})$.
\end{lemma}
\begin{proof}
We remark that the previous known results gathered in Theorem \ref{TH:knownresult} maintain their validity for the modified version of the N-S equations (see the appendix in \cite{FlandoliMaslowski1995ergodicity}). In order to adapt the new results in Section \ref{SEC:Mainresult}, we proceed by the finite dimensional approximations in equation \eqref{EQ:NSmodifiedapproximated}. The two differences from Section \ref{SEC:Mainresult} lie in the presence of the truncating function $\Theta_R$ and in the finite approximation of the Wiener noise.

We observe that the proofs of Lemmas \ref{LEM:atimaprioriferrario} and \ref{LEM:priola} are not modified if we add $\Theta_R$ in front of the non-linearity $B_n$ and replace $z$ by its finite dimensional approximation $z_n=\Pi_nz $. Indeed we can always employ the obvious estimates $\|\Theta_R\|_\infty=1$ and $\|A^\a z_n(t)\|\leq \|A^\a z(t)\|$ for all $t,\alpha\geq 0$ and $\mP-a.s.$
Thus we obtain analogous \textit{a priori} estimates for $v_n^R:=X^{R,n}-z_n$.

    In the first step from the proof of Theorem \ref{TH:ulunardi}, estimates for $v_n^R$ do not change if we replace $B_n$ with $\Theta_RB_n$ and $z$ with $z_n$. We conclude that a sub-sequence of $v_n^R$ converges in $C\big([0,T];D(A^\g)\big)$.
    
    As for the second step from the proof of Theorem \ref{TH:ulunardi}: we have $X^{R,n}=v^R_n+z_n$ and we need to prove that $z_n$ converges to $z$ in $C\big([0,T];D(A^\g)\big)$. This fact is true for a sub-sequence, 
    that we extract almost certainly through the Arzelà-Ascoli theorem. Indeed, the equi-continuity of the sequence derives from the regularities expressed in Theorem \ref{TH:z} and the following estimate for any $\beta>0$ with $\b+\g<1/4+ \eps $
    \[
    \|z_{n}(t)-z_{n}(s)\|_{D(A^\g)}\leq \|z(t)-z(s)\|_{D(A^\g)}\leq \|z\|_{C^{\b}([0,T];D(A^\g))}|t-s|^{\beta}.
    \]
    As for the relative compactness, we use the compact embedding $D(A^\a)\hookrightarrow D(A^\g)$ for $\g<\a<1/4+ \eps $ and the arbitrariness of $\g<1/4+ \eps $.
\end{proof}

Furthermore, it is classical to show that the solution to equation \eqref{EQ:NSmodified} is a Markov process (\textit{cf.} \cite{DaPZab1996ergo}), thus its Markov semigroup will be denoted by $\{P_t^R\}^{}_{t\geq 0}$.

 \begin{lemma}\label{LEM:PR}
    For any $G$ as in hypothesis \eqref{HP:H1} and for all $R,t>0$, there exists a constant $C_R(t)$ that satisfies  
    \[
    \big| P^R_t\varphi(x)-P^R_t\varphi(y)\big|\leq C_R(t)\|\varphi\|^{}_\infty\|x-y\|\qquad \forall \, \varphi\in C_b\big( D(A^{\frac14})\big), \ \forall\, x,y\in  D(A^{\frac14}).
    \]
\end{lemma}
\begin{proof}
    We fix $R>0$ and $x,y\in  D(A^{\frac14})$, we denote by $X^{R,n}(x),X^{R,n}(y)$ for all $n\in\mathbb N$ and by $X^{R}(x),X^{R}(y)$ the solutions to equations \eqref{EQ:NSmodifiedapproximated}  and \eqref{EQ:NSmodified} respectively, with starting point $x$ or $y$ respectively. 
    By Lemma \ref{LEM:allisfine} we have for all $\varphi\in C_b\big( D(A^{\frac14})\big)$ and all $t\geq 0$
    \[
    \varphi\big(X_t^{R,n}(x)\big)\longrightarrow \varphi\big(X_t^R(x)\big),\quad 
    \varphi\big(X_t^{R,n}(y)\big)\longrightarrow \varphi\big(X_t^R(y)\big)\qquad   \mP-a.s.
     \text{ as } n\to\infty.
    \]
    Therefore by 
    the dominated convergence theorem we have 
    \[
    \begin{cases}
        P_t^{R,n}\varphi(x)=\mE\, \varphi\big(X^{R,n}(x)\big)\longrightarrow \mE\, \varphi\big(X^{R}(x)\big)=P_t^{R}\varphi(x)\\
        P_t^{R,n}\varphi(y)=\mE\, \varphi\big(X^{R,n}(y)\big)\longrightarrow \mE\, \varphi\big(X^{R}(y)\big)=P_t^{R}\varphi(y)
    \end{cases}\qquad \forall\, t\geq0, \text{ as }n\to\infty.
    \]
    Thus, for all $t\geq 0$
    \begin{gather*}
    \big|P_t^R\varphi(x)-P_t^R\varphi(y)\big|\leq 
    \big|P_t^R\varphi(x)-P_t^{R,n}\varphi(x)\big|+\big|P_t^{R,n}\varphi(x)-P_t^{R,n}\varphi(y)\big|+\big|P_t^{R,n}\varphi(y)-P_t^R\varphi(y)\big|\\
    \leq \big|P_t^{R,n}\varphi(x)-P_t^{R,n}\varphi(y)\big|+o(1), \text{ as }n\to\infty.
    \end{gather*}
    By use of the uniform estimate in $n$ provided by Lemma \ref{LEM:PRn}, and taking the limit inferior as $n$ goes to infinity, we obtain
    \[
    \big| P^R_t\varphi(x)-P^R_t\varphi(y)\big|\leq C_R(t)\|\varphi\|^{}_\infty\|x-y\|\qquad \forall\, t\geq 0.
    \]
\end{proof}

\begin{lemma}\label{LEM:fer}
    For every $G$ as in hypothesis \eqref{HP:H1}, for every bounded and countable set $U\subset  D(A^{\frac14})$ and for all times $t>0$
    it holds 
    \[
    \lim_{R\to\infty} \sup_{
    x\in U}
    \|{P_t^R}^*\delta_x-P_t^*\delta_x\|_{TV}^{}=0.
    \]
\end{lemma}
\begin{proof}
    Let us fix $G, U$ as in the hypothesis and take $x\in U$, $t>0$ and $\varphi\in C_b\big( D(A^{\frac14})\big)$ with $\|\varphi\|^{}_{\infty}\leq 1$, then we need to prove that 
    \[
    \lim_{R\to\infty}|P_t^R\varphi(x)-P_t\varphi(x)|=0
    \]
    uniformly with respect to $x$ and $\varphi$. 
    We define
    \[
    \Lambda:=\Big\{\sup_{x\in U
    }\sup_{t\in[0,T]}\|A^{\frac14}X_t^R(x)-A^{\frac14}X_t(x)\|^2>0\Big\}\subset \Big\{\sup_{x\in U
    }\sup_{t\in[0,T]}\|A^{\frac14}X_t^R(x)\|^2>R\Big\},
    \]
    where the inclusion holds by definition of $\Theta_R$ and uniqueness of equation \eqref{EQ:NSmodified}. Observe that $\Lambda$ is measurable in $\Omega$, thanks to the continuity of $X(x),X^R(x):[0,T]\to D(A^{\frac14})$ and the fact that $U$ is countable. We also have by direct inspection that $\mi_{\Lambda^c}X^R_t(x)=\mi_{\Lambda^c}X_t(x)$ for all times and all $x\in U$. 
    Therefore
    \begin{gather*}
        \sup_{x\in U
        }|P_t^R\varphi(x)-P_t\varphi(x)|
        \leq \sup_{x\in U
        }\mE\Big[\big|\varphi(X^R_t(x))-\varphi(X_t(x))\big|(\mi_\Lambda+\mi_{\Lambda^c})\Big]
        \leq 2\, \mP(\Lambda)\\
        \leq 2 \, \mE\Big[\mi_{(R,+\infty)}\Big(\sup_{x\in U        }\sup_{t\in[0,T]}\|A^{\frac14}X_t^R(x)\|^2\Big)\Big].
    \end{gather*}
    We know by Lemma \ref{LEM:allisfine} that $\|A^{\frac14}X_t^R(x)\|$ is $\mP-a.s.$ uniformly bounded both in time, in $R$ and in $x$ in bounded sets of $D(A^{\frac14})$, thus we conclude by the dominated convergence theorem 
    \[
    \sup_{x\in U
    }\|{P_t^R}^*\delta_x-P_t^*\delta_x\|^{}_{TV}\longrightarrow0\qquad \text{ as }R\to\infty.
    \]
\end{proof}

 \begin{theorem}
     For any $G$ as in hypothesis \eqref{HP:H1}, the Markov semigroup $\{P_t\}^{}_{t\geq 0}$ associated to the solution of equation \eqref{EQ:NS} (\textit{cf.} Section \ref{SEC:P_tmu}) enjoys the \textit{(SF)} property on $ D(A^{\frac14})\hookrightarrow H$.
 \end{theorem}
 
 \begin{proof}
By Lemma \ref{LEM:PR} we have for all $R>0$, $t>0$ and $x,y\in  D(A^{\frac14})$:
\begin{align*}
    \big\|{P^R_t}^*\delta_x-{P^R_t}^*\delta_y\big\|_{TV}&:=\sup \left\{|P_t\varphi(x)-P^R_t\varphi(y)| : \varphi\in \mathcal B_b\big( D(A^{\frac14})\big), \|\varphi\|_{\infty}\leq 1\right\}\\
    &=\sup \left\{|P_t\varphi(x)-P^R_t\varphi(y)| : \varphi\in C_b\big( D(A^{\frac14})\big), \|\varphi\|_{\infty}\leq 1\right\}\leq C_R(t)\|x-y\|.
\end{align*}
We now fix $t> 0$,  $x\in  D(A^{\frac14})$ and take an arbitrary sequence $\{x_n\}_{n\in\mN}^{}$ bounded in $D(A^{\frac14})$ and converging to $x\in D(A^{\frac14})$ with respect to the norm of $H$.
For any $\epsilon>0$, Lemma \ref{LEM:fer} returns $R>0$ (which depends both on $\epsilon$ and on $\sup_{n\in\mN}\|A^{\frac14}x_n\|$) large enough such that 
\begin{align*}
|P_t\varphi(x)-P_t\varphi(x_n)|
&\leq |P_t\varphi(x)-P^R_t\varphi(x)|+|P^R_t\varphi(x)-P^R_t\varphi(x_n)|+|P^R_t\varphi(x_n)-P_t\varphi(x_n)|\\
&\leq \Big[\|P_t^*\delta_x-{P^R_t}^*\delta_x\|_{TV}
+\|{P^R_t}^*\delta_x-{P^R_t}^*\delta_{x_n}\|_{TV}
+\|P_t^*\delta_{x_n}-{P^R_t}^*\delta_{x_n}\|_{TV}\Big]\|\varphi\|^{}_\infty\\
&\leq C_R(t)\|\varphi\|^{}_\infty\|x-x_n\|+2\|\varphi\|^{}_\infty\epsilon.
\end{align*}
By taking the limit inferior as $n$ goes to infinity we find out that
\[
0\leq\liminf_{n\to\infty}|P_t\varphi(x)-P_t\varphi(x_n)|\leq 2\|\varphi\|^{}_\infty\epsilon,
\]
which gives the thesis by arbitrariness of $\epsilon>0.$
\end{proof}


\newpage
\appendix
\newcounter{appendixsection}
\newcounter{appendixequation}
\renewcommand{\theappendixsection}{\Alph{appendixsection}}

\theoremstyle{plain}
\newtheorem{theo}{Theorem}[appendixsection]
\newtheorem{lem}[theo]{Lemma}
\newtheorem{prop}[theo]{Proposition}

\section*{Appendix}
\setcounter{appendixsection}{1}  

We provide a modified version of the classical Gr\"onwall lemma, as presented in 
\cite[Section $1.2.1$]{henry1981geometric}.
\begin{lem}[Modified Gr\"onwall's lemma]\label{LEM:modifiedgronwall}
    Let us take $a,b\geq 0$, $T>0$ and $\a,\b\in[0,1)$. There exists a finite constant $M>0$ so that for any integrable function $u:(0,T)\to\R$ satisfying 
    \[
    0\leq u(t)\leq at^{-\a}+b\int_0^t(t-s)^{-\b}u(s)\d s \qquad \text{for }a.e.\ t\in(0,T),
    \]
we have 
\[
0\leq u(t)\leq aMt^{-\a}\qquad \text{for }a.e.\ t\in(0,T).
\]
\end{lem}

\begin{proof} [\textbf{Proof of Lemma \ref{LEM:aux}}]
    Let us assume that $P_t^*\delta_{x_0}(U)>0$ for some $t>0$, $x_0\in  D(A^\a) $ and $U\in \mathscr B_{ D(A^\a) }$. 
    We arbitrarily fix $x\in  D(A^\a) $ and $s> 0$, then the thesis is proved once we show that $P^*_s\delta_x(U)>0$. 
    
    First we suppose $s>t$. The $\textit{(SF)}$ property on $ D(A^\a) \hookrightarrow H$ applied at time $s-t$ and at point $x_0$ assures that there are $\epsilon,M>0$ such that $P_{s-t}\mi_U(y)>0$ for all $y\in B_\epsilon^H(x_0)\cap  B_M^{D(A^\a)} :=\{y\in  D(A^\a)  \ : \ \|A^\a y\|<M, \ \|y-x_0\|<\epsilon\}$. 
    Moreover $P_t^*\delta_x>0$ on open non-empty Borel sets in $D(A^\a)$ thanks to irreducibility in $ D(A^\a) $. Resort to the Chapman-Kolmogorov equation to conclude:
\[
P_s^*\delta_x(U)=\int_{ H }P_{s-t}\mi_U(y)\, P_t^*\delta_x(\!\d y)\geq \int\limits_{B_\epsilon^H(x_0)\cap  B_M^{D(A^\a)} }P_{s-t}\mi_U(y)\, P_t^*\delta_x(\!\d y)>0.
\]

Assume now $s\in (0,t]$ and choose $h\in (t-s,t)$ to write
\[
0<P_t^*\delta_{x_0}(U)=\int_{ H }P_{t-h}\mi_U(y)\, P_h^*\delta_{x_0}(\!\d y).
\]
Therefore, by density of $D(A^\a)$ in $H$, there exists $y_0\in  D(A^\a) $ such that $P_{t-h}\mi_U(y_0)>0$ and so we can repeat word for word the first step of the proof by substituting $x_0$ with $y_0$ and $t$ with $t-h<s$. We conclude once again that $P_s^*\delta_x(U)>0$, which completes the proof.
\end{proof}

\begin{prop}\label{PROP:casini}
Given $G\in \mathcal L(H)$ injective and such that $V\subset \Ran(G)$, then there exists $C>0$ such that 
    \[
    \big\|(\Pi_nGG^*\Pi_n)^{-1/2}x\big\|\leq C\|x\|^{}_V\qquad \forall\, x\in H_n.
    \]
\end{prop}
\begin{proof}
The proof is inspired by \cite[Section $6$]{FlandoliMaslowski1995ergodicity} and we present it for the sake of completeness.\\
     \textit{Step $1$.} Notice that $\Pi_nG:H\to\Pi_n\!\Ran(G)$ is bijective, thus its inverse operator is well-defined, yet generally unbounded $(\Pi_nG)^{-1}:\Pi_n\!\Ran(G)\to H$. From $V\subset \Ran(G)$ we have $H_n=\Pi_nV\subset\Pi_n\!\Ran(G)$, therefore it is well-defined and bounded $(\Pi_nG)^{-1}A^{-\frac12}:H_n\to H$. 

     \textit{Step $2$.} We show that $A^{\frac12}\Pi_nGG^*\Pi_nA^{\frac12}:H_n\to H_n$ is invertible with bounded inverse. Indeed we have
\begin{align*}        
&\quad\,\Big[\big((\Pi_nG)^{-1}A^{-\frac12}\big)^*(\Pi_nG)^{-1}A^{-\frac12}\Big]\Big[A^{\frac12}\Pi_nGG^*\Pi_nA^{\frac12}\Big]\\
        &=\big((\Pi_nG)^{-1}A^{-\frac12}\big)^*(A^{\frac12}\Pi_nG)^*\\
        &=\Big[A^{\frac12}\Pi_nG(\Pi_nG)^{-1}A^{-\frac12}\Big]^*=\id_{H_n}.
    \end{align*}
    This proves that $\big(A^{\frac12}\Pi_nGG^*\Pi_nA^{\frac12}\big)^{-1}=\big((\Pi_nG)^{-1}A^{-\frac12}\big)^*(\Pi_nG)^{-1}A^{-\frac12}$, which is bounded in $H_n$ by the previous step. Therefore there exists $C>0$ such that 
    \[    \big\|\big(A^{\frac12}\Pi_nGG^*\Pi_nA^{\frac12}\big)^{-1}x\big\|\leq C\|x\|\qquad\forall\,x\in H_n.
    \]
    
     \textit{Step $3$.} Let us now take $x\in H_n$ then
    \begin{align*}
        \big\|(\Pi_nGG^*\Pi_n)^{-1/2}x\big\|^2&=\big\langle (\Pi_nGG^*\Pi_n)^{-1}x,x\big\rangle\\
        &=\big\langle A^{\frac12}A^{-\frac12}(\Pi_nGG^*\Pi_n)^{-1}A^{-\frac12}A^{\frac12}x,x\big\rangle\\
        &=\big\langle (A^{\frac12}\Pi_nGG^*\Pi_nA^{\frac12})^{-1}A^{\frac12}x,A^{\frac12}x\big\rangle\\&
        \leq C\|A^{\frac12}x\|^2.
    \end{align*}
\end{proof}

\begin{prop}\label{PROP:estderi}
    For any $R>0$ and $t>0$ there exists a constant $C_R(t)>0$ such that for all $G$ as in hypothesis \eqref{HP:H1} it holds
    \[
    \int_0^t\big\|D_x[X_s^{R,n}(x)]h\big\|^2_V\d s \leq C_R(t)\|h\|^2\qquad \forall\, n\in\mathbb N, \ \forall\, h,x\in H_n, \ \mP-a.s.
    \]
\end{prop}
\begin{proof}
    The proof is inspired by the appendix in \cite{FlandoliMaslowski1995ergodicity} and we present it for the sake of completeness. We fix $\mP-a.s.$ $\omega\in\Omega$ and parameters $n\in\mN$, $R>0$. For all starting points $x\in H_n$ and times $t\geq 0$ we denote for simplicity $u(t)=X_t^{R,n}(x)(\omega)$ and we recall its definition:
    \[
    u(t)+\int_0^tAu(s)\d s +\int_0^t\Theta_R\big(\|A^{\frac14}u(s)\|^2\big)B_n(u(s))\d s =x+\Pi_nGW^n_t.
    \]
    Given the fact that this equation is set in the finite dimensional space $H_n$, spanned by the first $n$ eigenvectors of $A$, we take its components with respect to $\{e_k\}_{k=1}^n$ and consider it in $\R^n$. Therefore we perform the Jacobian matrix $D_x$ of $u=(u_1,\dots,u_n)^T$ with respect to the starting point $x=(x_1,\dots, x_n)^T\in H_n$ and apply it to a vector $h=(h_1,\dots,h_n)^T\in H_n$:
    \begin{equation}\label{EQ:matricial} 
    \begin{split}
    D_x[u(t)]h+\int_0^tAD_x[u(s)]h\d s +\int_0^t\Theta'_R\big(\|A^{\frac14}u(s)\|^2\big)    h\cdot\nabla_x\|A^{\frac14}u(s)\|^2\, B_n\big(u(s)\big)
    \d s\\
    +\int_0^t\Theta_R\big(\|A^{\frac14}u(s)\|^2\big)D_x\big[B_n\big(u(s)\big)\big]h\d s=h.\hspace{3cm}
    \end{split}\end{equation}
    We denote for brevity $U(t)=D_x[u(t)]h=\sum_{j=1}^nh_j\partial_{x_j}u(t)$, suppress the dependence on $t$, derive in time equation \eqref{EQ:matricial}, then take its product in $H_n$ with $U$:
    \begin{gather}\label{EQ:U}
    \Big\langle \frac{\d}{\d t}U,U\Big\rangle +\langle AU,U\rangle
    +\Big\langle\Theta'_R\big(\|A^{\frac14}u\|^2\big)h\cdot\nabla_x\|A^{\frac14}u\|^2\, B_n(u),U\Big\rangle 
        +\Big\langle D_x[B_n(u)]h,\Theta_R\big(\|A^{\frac14}u\|^2\big)U\Big\rangle=0.
    \end{gather}
    The first term in the left-hand side of equation \eqref{EQ:U} equals the time derivative of $\|U\|^2/2$, while the second is equal to $\|U\|_{V}^2$. As for the third one, we start by evaluating the components of the gradient vector. For $j=1,\dots,n$ it holds
    \[
    \partial_{x_j} \|A^{\frac14}u^{}\|^2=\int_{{\mathcal D}}\partial_{x_j} |A^{1/4}u^{}|^2\d\mathscr L^2
    =2\int_{{\mathcal D}}A^{\frac14}u\cdot A^{\frac14}\partial_{x_j}u\d\mathscr L^2=2\big\langle A^{\frac14}u,A^{\frac14}\partial_{x_j}u\big\rangle,
    \]
    that implies 
    \[
    h\cdot\nabla_x\|A^{\frac14}u\|^2
    =2\langle A^{\frac14}u, A^{\frac14}U\rangle,
    \]
    thus, by the Cauchy-Schwarz inequality and Lemma \ref{LEM:bB}
    \begin{align*}
\Big|\Big\langle\Theta'_R\big(\|A^{\frac14}u\|^2\big)h\cdot\nabla_x\|A^{\frac14}u\|^2\, B_n(u),U\Big\rangle \Big|
    &\leq 2\|\Theta'_R\|^{}_\infty\big|\big\langle A^{\frac14}u,A^{\frac14}U\big\rangle\big|\, \big|\big\langle B_n(u),U\big\rangle \big|\\
    &\leq 2c_0\|\Theta'_R\|^{}_\infty\|A^{\frac14}u\|^3\|A^{\frac14}U\| \|U\|^{}_V.
    \end{align*}
    Analogously 
    \begin{align*}\Big\langle D_x[B_n(u)]h,\Theta_R\big(\|A^{\frac14}u\|^2\big)U\Big\rangle 
    &= \sum_{j=1}^nh_j\big\langle \partial_{x_j}B_n(u),\Theta_R\big(\|A^{\frac14}u\|^2\big)U\big\rangle \\
    &=\sum_{j=1}^nh_jb_n\big(\partial_{x_j}u,u,\Theta_R\big(\|A^{\frac14}u\|^2\big)U\big) +b_n\big(u,\partial_{x_j}u,\Theta_R\big(\|A^{\frac14}u\|^2\big)U\big)\\
    &=b_n\big(U,u,\Theta_R\big(\|A^{\frac14}u\|^2\big)U\big)+b_n\big(u,U,\Theta_R\big(\|A^{\frac14}u\|^2\big)U\big),
    \end{align*}
which entails, by Lemma \ref{LEM:bB} and $\|\Theta_R\|^{}_{\infty}=1$
\begin{align*}\Big|\Big\langle D_x\big[B_n(u)\big]h,\Theta_R\big(\|A^{\frac14}u\|^2\big)U\Big\rangle \Big|\leq 
2c_0\|A^{\frac14}U\|\|A^{\frac14}u\| \|U\|^{}_V.
    \end{align*}
    By plugging all terms inside equation \eqref{EQ:matricial}, we reach the following estimate
    \begin{align*}
    \frac{1}{2}\frac{\d}{\d t}\|U\|^2+\|U\|^2_V
    &\leq 2c_0\big(1\vee\|\Theta'_R\|_\infty\big)\|A^{\frac14}U\|\|U\|^{}_V\big(\|A^{\frac14}u\|^3+\|A^{\frac14}u\|\big)\\
    &\leq K_R\|U\|^{1/2}\,\|U\|^{1/2}_V\|U\|^{}_V\\
    &\leq \frac14 K_R^4\|U\|^2+\frac{3}{4}\|U\|^2_V,
    \end{align*}
    where we employed the interpolation inequality $\|A^{\frac14}U\|\leq \|U\|^{1/2}\,\|U\|^{1/2}_V$ (see Lemma \ref{LEM:sobolevinterpolation}), introduced the function 
    \[
K_R(t):=2c_0\big(1\vee\|\Theta'_R\|_\infty\big)\Big(\sup_{n\in\mN}\|A^{\frac14}u(t)\|^3+\sup_{n\in\mN}\|A^{\frac14}u(t)\|\Big),
    \]
    and iterated the Young inequality. 
    We integrate over the interval $[0,t]$ recalling that $U(0)=D_x[u(0)]h~=~h$
    \begin{equation}\label{EQ:Uappe}
    \|U(t)\|^2+\frac{1}{2}\int_0^t\|U(s)\|^2_V\d s \leq \|h\|^2+\int_0^tK^4_R(s)\|U(s)\|^2\d s,
    \end{equation}
    apply Gr\"onwall's lemma and obtain
    \[
    \|U(t)\|^2\leq \|h\|^2e^{K^4_R(t)}.
    \]
    Inserting back this estimate in equation \eqref{EQ:Uappe}, we reach the thesis for a constant $C_R(t)>0$
    \[
    \int_0^t\|U(s)\|^2_V\d s \leq 2\|h\|^2\bigg[1+\int_0^tK^4_R(s)e^{K^4_R(s)}\d s\bigg]=\|h\|^2C_R(t).
    \]
\end{proof}

\begin{lem}\label{LEM:sigmainduced}
For all $\alpha>0$ the $\sigma$-algebras $\mathscr{B}_{D(A^\a)}$ and $\mathscr{B}_H \cap D(A^\a)$ coincide on $D(A^\a)$.
\end{lem}
\begin{proof}
    Let $\iota: D(A^\a)  \to H$ denote the linear continuous embedding and $\tau_{ D(A^\a)  }, \tau_H$ denote the topologies on $ D(A^\a) , H$ induced by the norms $\|\cdot\|^{}_{ D(A^\a)  }$ and $\|\cdot\|$, respectively. By continuity of $\iota$, we have $\iota^{-1}(\tau_H)\subset \tau_{ D(A^\a) }$, which means that the topology $\tau_{ D(A^\a)  }$ on $ D(A^\a)  $ is finer then the topology induced from $H$ via $\iota$. As a consequence
    \begin{equation*}
     \mathscr B_{ D(A^\a)  }=\sigma\big(\tau_{ D(A^\a)  }\big)\supset \sigma\big(\iota^{-1}(\tau_H)\big)=\iota^{-1}\big(\sigma(\tau_H)\big)=\iota^{-1}(\mathscr B_H)=\mathscr B_H\cap  D(A^\a)  ,
     \end{equation*}
     to wit the $\sigma$-algebra induced from $H$ and denoted by $\mathscr B_H\cap  D(A^\a) $ is a sub-$\sigma$-algebra of $\mathscr B_{ D(A^\a) }$. 
     
     In order to invert the above inclusion we can not proceed in the same way. Should indeed the two topologies $\tau_{ D(A^\a)  }$,$\tau_H$ coincide, then the respective norms would be equivalent, which is false. Nevertheless, if we were able to prove that the open sets in $D(A^\a)$ are Borel sets in $H$, then we would easily reach the thesis by 
     \[
     \mathscr B_{ D(A^\a) }=\sigma\big(\tau_{ D(A^\a) }\big)\subset \sigma\big(\mathscr B_H\cap{ D(A^\a) }\big)=\mathscr B_H\cap D(A^\a) .
     \]
     In order to show that $\tau_{ D(A^\a) }\subset \mathscr B_H\cap { D(A^\a) }$, it is sufficient to prove that the $\|\cdot\|_{ D(A^\a) }$-norm on $ D(A^\a) $ is a $\mathscr B_H$-measurable function. 
     Let $\Gamma\subset  D(A^\a) $ be countable and dense in $H$ (such collection exists because $H$ is separable and $ D(A^\a) $ is dense in $H$). We have
    \[
\|x\|^{}_{ D(A^\a) }=\|A^{\a}x\|=\sup_{y\in H\setminus \{0\}}\frac{|\langle A^{\a}x,y\rangle|}{\|y\|}
=\sup_{y\in \Gamma\setminus \{0\}}\frac{|\langle A^{\a}x,y\rangle|}{\|y\|}
=\sup_{y\in \Gamma\setminus \{0\}}\frac{|\langle x,A^{\a}y\rangle|}{\|y\|}
\qquad \forall \, x\in  D(A^\a) ,
\]
where the second equality follows by the characterization of the norm in Hilbert spaces, the third by the density of $\Gamma$ in $H$ and the continuity in $H\setminus\{0\}$ of $y\mapsto {|\langle A^{\a}x,y\rangle|}/{\|y\|}$, while the last equality follows by self-adjointness of $A^{\a}$ (see Section \ref{SEC:operators}). For every $y\in \Gamma\setminus \{0\}$ we define the function 
\[
p_y:H\to \R:x\mapsto p_y(x)=\frac{\langle x, A^{\a}y\rangle}{\|y\|},
\]
and we observe that it is a linear and bounded functional, thus continuous, thus $\mathscr B_H$-measurable. 
We therefore conclude that $\|\cdot\|^{}_{ D(A^\a) }=\sup_{y\in \Gamma\setminus \{0\}}|p_y(\cdot)|$ is $\mathscr B_H$-measurable as it is the supremum of a countable family of $\mathscr{B}_H$-measurable functions.
\end{proof}


\noindent \textbf{Acknowledgements:} 
We express our sincere gratitude to Professor Enrico Priola for his insightful contributions and stimulating discussions on the subject. This research activity was carried out as part of the PRIN 2022 project $``$Noise in fluid dynamics and related models$"$. The author is also member of the $``$Gruppo Nazionale per l'Analisi Matematica, la Probabilità e le loro Applicazioni (GNAMPA)$"$, which is part of the  $``$Istituto Nazionale di Alta Matematica (INdAM)$"$.\\[3mm]


\printbibliography[heading=bibintoc]

@article {seidler1997ergodic,
    AUTHOR = {Seidler, Jan},
     TITLE = {Ergodic Behaviour of Stochastic Parabolic Equations},
   JOURNAL = {Czechoslovak Math. J.},
  FJOURNAL = {Czechoslovak Mathematical Journal},
    VOLUME = {47(122)},
      YEAR = {1997},
    NUMBER = {2},
     PAGES = {277--316},
      ISSN = {0011-4642,1572-9141},
   MRCLASS = {60H10 (35K99 35R60 60J35)},
  MRNUMBER = {1452421},
MRREVIEWER = {Isabel\ Sim\~{a}o},
       DOI = {10.1023/A:1022821729545},
URL = {https://doi.org/10.1023/A:1022821729545},
}

@book {vishik1988mathematical,
    AUTHOR = {Vishik, M. J. and Fursikov, A. V.},
     TITLE = {Mathematical Problems of Statistical Hydromechanics},
    SERIES = {Mathematics and its Applications (Soviet Series)},
    VOLUME = {9},
      NOTE = {Translated from the 1980 Russian original [MR0591678] by D. A.
              Leites},
 PUBLISHER = {Kluwer Academic Publishers Group, Dordrecht},
      YEAR = {1988},
     PAGES = {vii+576},
      ISBN = {978-94-010-7137-6},
   MRCLASS = {35Q30 (01A75 58D25 76Dxx)},
  MRNUMBER = {3444271},
       DOI = {10.1007/978-94-009-1423-0},
URL = {https://doi.org/10.1007/978-94-009-1423-0},
}

@article {dinezza2012hitchhikers,
    AUTHOR = {Di Nezza, Eleonora and Palatucci, Giampiero and Valdinoci,
              Enrico},
     TITLE = {Hitchhiker's guide to the fractional {S}obolev spaces},
   JOURNAL = {Bull. Sci. Math.},
  FJOURNAL = {Bulletin des Sciences Math\'{e}matiques},
    VOLUME = {136},
      YEAR = {2012},
    NUMBER = {5},
     PAGES = {521--573},
      ISSN = {0007-4497,1952-4773},
   MRCLASS = {46E35 (35A23 35S05 35S30)},
  MRNUMBER = {2944369},
MRREVIEWER = {Lanzhe\ Liu},
       DOI = {10.1016/j.bulsci.2011.12.004},
URL = {https://doi.org/10.1016/j.bulsci.2011.12.004},
}

@article {caetano1998eigenvalue,
    AUTHOR = {Caetano, Ant\'{o}nio M.},
     TITLE = {Eigenvalue Asymptotics of the Stokes Operator for Fractal Domains},
   JOURNAL = {Proc. London Math. Soc. (3)},
  FJOURNAL = {Proceedings of the London Mathematical Society. Third Series},
    VOLUME = {76},
      YEAR = {1998},
    NUMBER = {3},
     PAGES = {579--602},
      ISSN = {0024-6115,1460-244X},
   MRCLASS = {35P20 (35J99 35Q30 47F05 58G25)},
  MRNUMBER = {1620496},
MRREVIEWER = {Russell\ M.\ Brown},
       DOI = {10.1112/S0024611598000331},
URL = {https://doi.org/10.1112/S0024611598000331},
}

@article {metivier1978valeurspropres,
    AUTHOR = {M\'{e}tivier, Guy},
     TITLE = {Valeurs propres d'op\'{e}rateurs d\'{e}finis par la
              restriction de syst\`emes variationnels \`a des sous-espaces},
   JOURNAL = {J. Math. Pures Appl. (9)},
  FJOURNAL = {Journal de Math\'{e}matiques Pures et Appliqu\'{e}es.
              Neuvi\`eme S\'{e}rie},
    VOLUME = {57},
      YEAR = {1978},
    NUMBER = {2},
     PAGES = {133--156},
      ISSN = {0021-7824,1776-3371},
   MRCLASS = {35P20 (47A70)},
  MRNUMBER = {505900},
MRREVIEWER = {Joel\ Spruck},
}

@article {albeverio2004uniqueness,
    AUTHOR = {Albeverio, S. and Ferrario, B.},
     TITLE = {Uniqueness of solutions of the stochastic {N}avier-{S}tokes
              equation with invariant measure given by the enstrophy},
   JOURNAL = {Ann. Probab.},
  FJOURNAL = {The Annals of Probability},
    VOLUME = {32},
      YEAR = {2004},
    NUMBER = {2},
     PAGES = {1632--1649},
      ISSN = {0091-1798,2168-894X},
   MRCLASS = {60H15 (35Q30 35R60 60G17 76D03 76D05 76M35)},
  MRNUMBER = {2060312},
MRREVIEWER = {Isamu\ D\^{o}ku},
       DOI = {10.1214/009117904000000379},
URL = {https://doi.org/10.1214/009117904000000379},
}

@article {GubinelliJara2013Regularization,
    AUTHOR = {Gubinelli, M. and Jara, M.},
     TITLE = {Regularization by noise and stochastic {B}urgers equations},
   JOURNAL = {Stoch. Partial Differ. Equ. Anal. Comput.},
  FJOURNAL = {Stochastic Partial Differential Equations. Analysis and
              Computations},
    VOLUME = {1},
      YEAR = {2013},
    NUMBER = {2},
     PAGES = {325--350},
      ISSN = {2194-0401,2194-041X},
   MRCLASS = {35R60 (35B10 35Q30 35Q53 35R11 60H15 82B41)},
  MRNUMBER = {3327509},
MRREVIEWER = {Anne-Sophie\ de Suzzoni},
       DOI = {10.1007/s40072-013-0011-5},
URL = {https://doi.org/10.1007/s40072-013-0011-5},
}

@article {dapratodebussche2002two,
    AUTHOR = {Da Prato, Giuseppe and Debussche, Arnaud},
     TITLE = {Two-dimensional {N}avier-{S}tokes equations driven by a
              space-time white noise},
   JOURNAL = {J. Funct. Anal.},
  FJOURNAL = {Journal of Functional Analysis},
    VOLUME = {196},
      YEAR = {2002},
    NUMBER = {1},
     PAGES = {180--210},
      ISSN = {0022-1236,1096-0783},
   MRCLASS = {35Q30 (35R60 60H40 76D05 76M35)},
  MRNUMBER = {1941997},
MRREVIEWER = {Peter\ E.\ Kloeden},
       DOI = {10.1006/jfan.2002.3919},
URL = {https://doi.org/10.1006/jfan.2002.3919},
}

@article {bismut1981martingales,
    AUTHOR = {Bismut, Jean-Michel},
     TITLE = {Martingales, the {M}alliavin Calculus and Hypoellipticity
              Under General {H}\"{o}rmander's Conditions},
   JOURNAL = {Z. Wahrsch. Verw. Gebiete},
  FJOURNAL = {Zeitschrift f\"{u}r Wahrscheinlichkeitstheorie und Verwandte
              Gebiete},
    VOLUME = {56},
      YEAR = {1981},
    NUMBER = {4},
     PAGES = {469--505},
      ISSN = {0044-3719},
   MRCLASS = {60H15 (35H05 58G32 60G48)},
  MRNUMBER = {621660},
MRREVIEWER = {Jean-Louis\ Ducourtioux},
       DOI = {10.1007/BF00531428},
URL = {https://doi.org/10.1007/BF00531428},
}

@book {cerrai2001second,
    AUTHOR = {Cerrai, Sandra},
     TITLE = {Second order {PDE}'s in finite and infinite dimension},
    SERIES = {Lecture Notes in Mathematics},
    VOLUME = {1762},
      NOTE = {A probabilistic approach},
 PUBLISHER = {Springer-Verlag, Berlin},
      YEAR = {2001},
     PAGES = {x+330},
      ISBN = {3-540-42136-X},
   MRCLASS = {35R60 (35K57 47H20 47N20 60H10 60H15 60J35)},
  MRNUMBER = {1840644},
MRREVIEWER = {Krystyna\ Twardowska},
       DOI = {10.1007/b80743},
URL = {https://doi.org/10.1007/b80743},
}

@article {bensoussantemam1973equations,
    AUTHOR = {Bensoussan, A. and Temam, R.},
     TITLE = {\'{E}quations Stochastiques du Type {N}avier-{S}tokes},
   JOURNAL = {J. Functional Analysis},
  FJOURNAL = {Journal of Functional Analysis},
    VOLUME = {13},
      YEAR = {1973},
     PAGES = {195--222},
      ISSN = {0022-1236},
   MRCLASS = {60H15 (35Q10)},
  MRNUMBER = {348841},
MRREVIEWER = {J.\ A.\ Goldstein},
       DOI = {10.1016/0022-1236(73)90045-1},
URL = {https://doi.org/10.1016/0022-1236(73)90045-1},
}

@book {renardy2004introduction,
    AUTHOR = {Renardy, Michael and Rogers, Robert C.},
     TITLE = {An Introduction to Partial Differential Equations},
    SERIES = {Texts in Applied Mathematics},
    VOLUME = {13},
   EDITION = {Second},
 PUBLISHER = {Springer-Verlag, New York},
      YEAR = {2004},
     PAGES = {xiv+434},
      ISBN = {0-387-00444-0},
   MRCLASS = {35-01 (46N20 47F05 47N20)},
  MRNUMBER = {2028503},
}

@book {DaPZab1996ergo,
    AUTHOR = {Da Prato, G. and Zabczyk, J.},
     TITLE = {Ergodicity for Infinite Dimensional Systems},
    SERIES = {London Mathematical Society Lecture Note Series},
    VOLUME = {229},
 PUBLISHER = {Cambridge University Press, Cambridge},
      YEAR = {1996},
     PAGES = {xii+339},
      ISBN = {0-521-57900-7},
   MRCLASS = {60H15 (28D05 60J25 60J35)},
  MRNUMBER = {1417491},
MRREVIEWER = {Bohdan\ Maslowski},
       DOI = {10.1017/CBO9780511662829},
URL = {https://doi.org/10.1017/CBO9780511662829},
}

@book {da2014stochastic,
    AUTHOR = {Da Prato, Giuseppe and Zabczyk, Jerzy},
     TITLE = {Stochastic Equations in Infinite Dimensions},
    SERIES = {Encyclopedia of Mathematics and its Applications},
    VOLUME = {152},
   EDITION = {Second},
 PUBLISHER = {Cambridge University Press, Cambridge},
      YEAR = {2014},
     PAGES = {xviii+493},
      ISBN = {978-1-107-05584-1},
   MRCLASS = {60H15 (34F05 34Gxx)},
  MRNUMBER = {3236753},
MRREVIEWER = {David\ Nualart},
       DOI = {10.1017/CBO9781107295513},
URL = {https://doi.org/10.1017/CBO9781107295513},
}

@book {henry1981geometric,
    AUTHOR = {Henry, Daniel},
     TITLE = {Geometric Theory of Semilinear Parabolic Equations},
    SERIES = {Lecture Notes in Mathematics},
    VOLUME = {840},
 PUBLISHER = {Springer-Verlag, Berlin-New York},
      YEAR = {1981},
     PAGES = {iv+348},
      ISBN = {3-540-10557-3},
   MRCLASS = {35K55 (34G20 58D25)},
  MRNUMBER = {610244},
MRREVIEWER = {J.\ A.\ Goldstein},
}

@article {ferrario1997ergodic,
    AUTHOR = {Ferrario, Benedetta},
     TITLE = {Ergodic Results for Stochastic Navier-Stokes Equation},
   JOURNAL = {Stochastics Stochastics Rep.},
  FJOURNAL = {Stochastics and Stochastics Reports},
    VOLUME = {60},
      YEAR = {1997},
    NUMBER = {3-4},
     PAGES = {271--288},
      ISSN = {1045-1129},
   MRCLASS = {60H15 (35Q30 35R60 47D07 76D05 76M35)},
  MRNUMBER = {1467721},
       DOI = {10.1080/17442509708834110},
URL = {https://doi.org/10.1080/17442509708834110},
}

@article {ferrario1999stochastic,
    AUTHOR = {Ferrario, Benedetta},
     TITLE = {Stochastic Navier-Stokes Equations: Analysis of the Noise to Have a Unique Invariant Measure},
   JOURNAL = {Ann. Mat. Pura Appl. (4)},
  FJOURNAL = {Annali di Matematica Pura ed Applicata. Serie Quarta},
    VOLUME = {177},
      YEAR = {1999},
     PAGES = {331--347},
      ISSN = {0003-4622},
   MRCLASS = {60H15 (35Q30 60J35 76D05)},
  MRNUMBER = {1747638},
MRREVIEWER = {Jan\ I.\ Seidler},
       DOI = {10.1007/BF02505916},
URL = {https://doi.org/10.1007/BF02505916},
}

@article {ferrario2003uniqueness,
    AUTHOR = {Ferrario, Benedetta},
     TITLE = {Uniqueness Result for the 2D Navier-Stokes Equation with Additive Noise},
   JOURNAL = {Stoch. Stoch. Rep.},
  FJOURNAL = {Stochastics and Stochastics Reports},
    VOLUME = {75},
      YEAR = {2003},
    NUMBER = {6},
     PAGES = {435--442},
      ISSN = {1045-1129,1029-0346},
   MRCLASS = {76D03 (60G17 60H15 76D05 76M35)},
  MRNUMBER = {2029616},
MRREVIEWER = {Anna\ Karczewska},
       DOI = {10.1080/10451120310001644485},
URL = {https://doi.org/10.1080/10451120310001644485},
}

@article {maslowski1993probability,
    AUTHOR = {Maslowski, Bohdan},
     TITLE = {On Probability Distributions of Solutions of Semilinear Stochastic Evolution Equations},
   JOURNAL = {Stochastics Stochastics Rep.},
  FJOURNAL = {Stochastics and Stochastics Reports},
    VOLUME = {45},
      YEAR = {1993},
    NUMBER = {1-2},
     PAGES = {17--44},
      ISSN = {1045-1129},
   MRCLASS = {60H15 (34F05 34G20)},
  MRNUMBER = {1277360},
MRREVIEWER = {Ralf\ Manthey},
       DOI = {10.1080/17442509308833854},
URL = {https://doi.org/10.1080/17442509308833854},
}

@article {giga1983weak,
    AUTHOR = {Giga, Yoshikazu},
     TITLE = {Weak and Strong Solutions of the Navier-Stokes Initial Value Problem},
   JOURNAL = {Publ. Res. Inst. Math. Sci.},
  FJOURNAL = {Kyoto University. Research Institute for Mathematical
              Sciences. Publications},
    VOLUME = {19},
      YEAR = {1983},
    NUMBER = {3},
     PAGES = {887--910},
      ISSN = {0034-5318,1663-4926},
   MRCLASS = {35Q10},
  MRNUMBER = {723454},
MRREVIEWER = {Giovanni\ P.\ Galdi},
       DOI = {10.2977/prims/1195182014},
URL = {https://doi.org/10.2977/prims/1195182014},
}

@article {giga1986solutions,
    AUTHOR = {Giga, Yoshikazu},
     TITLE = {Solutions for Semilinear Parabolic Equations in $L^p$ and Regularity of Weak Solutions of the Navier-Stokes System},
   JOURNAL = {J. Differential Equations},
  FJOURNAL = {Journal of Differential Equations},
    VOLUME = {62},
      YEAR = {1986},
    NUMBER = {2},
     PAGES = {186--212},
      ISSN = {0022-0396,1090-2732},
   MRCLASS = {35K55 (35Q10 76D05)},
  MRNUMBER = {833416},
MRREVIEWER = {Ky\^{u}ya\ Masuda},
       DOI = {10.1016/0022-0396(86)90096-3},
URL = {https://doi.org/10.1016/0022-0396(86)90096-3},
}

@article {Soboleskii1959onnonstationary,
    AUTHOR = {Sobolevskiĭ, P. E.},
     TITLE = {Non-Stationary Equations of Viscous Fluid Dynamics},
   JOURNAL = {Dokl. Akad. Nauk SSSR},
  FJOURNAL = {Doklady Akademii Nauk SSSR},
    VOLUME = {128},
      YEAR = {1959},
     PAGES = {45--48},
      ISSN = {0002-3264},
   MRCLASS = {35.00},
  MRNUMBER = {110895},
MRREVIEWER = {N.\ D.\ Kazarinoff},
}

@article {kielhofer1980global,
    AUTHOR = {Kielh\"{o}fer, Hansj\"{o}rg},
     TITLE = {Global Solutions of Semilinear Evolution Equations Satisfying an Energy Inequality},
   JOURNAL = {J. Differential Equations},
  FJOURNAL = {Journal of Differential Equations},
    VOLUME = {36},
      YEAR = {1980},
    NUMBER = {2},
     PAGES = {188--222},
      ISSN = {0022-0396,1090-2732},
   MRCLASS = {34G20},
  MRNUMBER = {574336},
MRREVIEWER = {T.\ Kato},
       DOI = {10.1016/0022-0396(80)90063-7},
URL = {https://doi.org/10.1016/0022-0396(80)90063-7},
}

@book {temam2001navier,
    AUTHOR = {Temam, Roger},
     TITLE = {Navier-Stokes Equations. Theory and Numerical Analysis},
    SERIES = {Studies in Mathematics and its Applications},
    VOLUME = {Vol. 2},
 PUBLISHER = {North-Holland Publishing Co., Amsterdam-New York-Oxford},
      YEAR = {1977},
     PAGES = {x+500},
      ISBN = {0-7204-2840-8},
   MRCLASS = {35Q10 (65P05 76.35)},
  MRNUMBER = {609732},
MRREVIEWER = {Marjorie\ McCracken},
}

@article {gigamiyakawa1985solutions,
    AUTHOR = {Giga, Yoshikazu and Miyakawa, Tetsuro},
     TITLE = {Solutions in $L_r$ of the Navier-Stokes Initial Value Problem},
   JOURNAL = {Arch. Rational Mech. Anal.},
  FJOURNAL = {Archive for Rational Mechanics and Analysis},
    VOLUME = {89},
      YEAR = {1985},
    NUMBER = {3},
     PAGES = {267--281},
      ISSN = {0003-9527},
   MRCLASS = {35Q10},
  MRNUMBER = {786550},
MRREVIEWER = {Alberto\ Valli},
       DOI = {10.1007/BF00276875},
URL = {https://doi.org/10.1007/BF00276875},
}

@article {FlandoliMaslowski1995ergodicity,
    AUTHOR = {Flandoli, Franco and Maslowski, Bohdan},
     TITLE = {Ergodicity of the {$2$}-{D} {N}avier-{S}tokes Equation Under
              Random Perturbations},
   JOURNAL = {Comm. Math. Phys.},
  FJOURNAL = {Communications in Mathematical Physics},
    VOLUME = {172},
      YEAR = {1995},
    NUMBER = {1},
     PAGES = {119--141},
      ISSN = {0010-3616,1432-0916},
   MRCLASS = {35R60 (35Q30 60H15 60H30 76D05)},
  MRNUMBER = {1346374},
MRREVIEWER = {Peter\ E.\ Kloeden},
       URL = {http://projecteuclid.org/euclid.cmp/1104273961},
}

@article {ferrario20192dnavier,
    AUTHOR = {Ferrario, Benedetta and Olivera, Christian},
     TITLE = {2{D} {N}avier-{S}tokes equation with cylindrical fractional
              {B}rownian noise},
   JOURNAL = {Ann. Mat. Pura Appl. (4)},
  FJOURNAL = {Annali di Matematica Pura ed Applicata. Series IV},
    VOLUME = {198},
      YEAR = {2019},
    NUMBER = {3},
     PAGES = {1041--1067},
      ISSN = {0373-3114,1618-1891},
   MRCLASS = {60H15 (35Q30 35R11 35R60 60H30 76D06)},
  MRNUMBER = {3954404},
MRREVIEWER = {Dimitra\ C.\ Antonopoulou},
       DOI = {10.1007/s10231-018-0809-x},
URL = {https://doi.org/10.1007/s10231-018-0809-x},
}

@book {pazy1992semigroups,
    AUTHOR = {Pazy, A.},
     TITLE = {Semigroups of Linear Operators and Applications to Partial Differential Equations},
    SERIES = {Applied Mathematical Sciences},
    VOLUME = {44},
 PUBLISHER = {Springer-Verlag, New York},
      YEAR = {1983},
     PAGES = {viii+279},
      ISBN = {0-387-90845-5},
   MRCLASS = {47D05 (34Gxx 35Fxx 35Gxx 47H20)},
  MRNUMBER = {710486},
MRREVIEWER = {H.\ O.\ Fattorini},
       DOI = {10.1007/978-1-4612-5561-1},
URL = {https://doi.org/10.1007/978-1-4612-5561-1},
}

@book {temam1995navier,
    AUTHOR = {Temam, Roger},
     TITLE = {Navier-{S}tokes Equations and Nonlinear Functional Analysis},
    SERIES = {CBMS-NSF Regional Conference Series in Applied Mathematics},
    VOLUME = {66},
   EDITION = {Second},
 PUBLISHER = {Society for Industrial and Applied Mathematics (SIAM),
              Philadelphia, PA},
      YEAR = {1995},
     PAGES = {xiv+141},
      ISBN = {0-89871-340-4},
   MRCLASS = {35Q30 (34G20 46N20 47H20 47N20 76D05)},
  MRNUMBER = {1318914},
MRREVIEWER = {Vadim\ Bondarevsky},
       DOI = {10.1137/1.9781611970050},
URL = {https://doi.org/10.1137/1.9781611970050},
}

@incollection {Kupiainen2011Ergodicity,
    AUTHOR = {Kupiainen, Antti},
     TITLE = {Ergodicity of two dimensional turbulence (after {H}airer and
              {M}attingly)},
      NOTE = {S\'{e}minaire Bourbaki. Vol. 2009/2010. Expos\'{e}s
              1012--1026},
   JOURNAL = {Ast\'{e}risque},
  FJOURNAL = {Ast\'{e}risque},
    NUMBER = {339},
      YEAR = {2011},
     PAGES = {Exp. No. 1016, vii, 137--156},
      ISSN = {0303-1179,2492-5926},
      ISBN = {978-2-85629-326-3},
   MRCLASS = {60H15 (76F02 76F55)},
  MRNUMBER = {2906352},
MRREVIEWER = {Peter\ E.\ Kloeden},
}

@article {hairer2006ergodicity,
    AUTHOR = {Hairer, Martin and Mattingly, Jonathan C.},
     TITLE = {Ergodicity of the 2{D} {N}avier-{S}tokes equations with
              degenerate stochastic forcing},
   JOURNAL = {Ann. of Math. (2)},
  FJOURNAL = {Annals of Mathematics. Second Series},
    VOLUME = {164},
      YEAR = {2006},
    NUMBER = {3},
     PAGES = {993--1032},
      ISSN = {0003-486X,1939-8980},
   MRCLASS = {37L55 (35Q30 35R60 60H15 76D05 76M35)},
  MRNUMBER = {2259251},
MRREVIEWER = {Hakima\ Bessaih},
       DOI = {10.4007/annals.2006.164.993},
URL = {https://doi.org/10.4007/annals.2006.164.993},
}

@article {elworthy1994formulae,
    AUTHOR = {Elworthy, K. D. and Li, X.-M.},
     TITLE = {Formulae for the Derivatives of Heat Semigroups},
   JOURNAL = {J. Funct. Anal.},
  FJOURNAL = {Journal of Functional Analysis},
    VOLUME = {125},
      YEAR = {1994},
    NUMBER = {1},
     PAGES = {252--286},
      ISSN = {0022-1236,1096-0783},
   MRCLASS = {60H10 (35K99 35R60 58G32)},
  MRNUMBER = {1297021},
MRREVIEWER = {R\'{e}mi\ L\'{e}andre},
       DOI = {10.1006/jfan.1994.1124},
URL = {https://doi.org/10.1006/jfan.1994.1124},
}

@article {katofujita1964navier,
    AUTHOR = {Fujita, Hiroshi and Kato, Tosio},
     TITLE = {On the {N}avier-{S}tokes Initial Value Problem. {I}},
   JOURNAL = {Arch. Rational Mech. Anal.},
  FJOURNAL = {Archive for Rational Mechanics and Analysis},
    VOLUME = {16},
      YEAR = {1964},
     PAGES = {269--315},
      ISSN = {0003-9527},
   MRCLASS = {35.79},
  MRNUMBER = {166499},
MRREVIEWER = {P.\ C.\ Fife},
       DOI = {10.1007/BF00276188},
URL = {https://doi.org/10.1007/BF00276188},
}

@article {katofujita1962nonstationary,
    AUTHOR = {Kato, Tosio and Fujita, Hiroshi},
     TITLE = {On the Nonstationary Navier-Stokes System},
   JOURNAL = {Rend. Sem. Mat. Univ. Padova},
  FJOURNAL = {Rendiconti del Seminario Matematico della Universit\`a di
              Padova. The Mathematical Journal of the University of Padova},
    VOLUME = {32},
      YEAR = {1962},
     PAGES = {243--260},
      ISSN = {0041-8994},
   MRCLASS = {35.79},
  MRNUMBER = {142928},
MRREVIEWER = {C.\ L.\ Dolph},
       URL = {http://www.numdam.org/item?id=RSMUP_1962__32__243_0},
}

@book {baldi2017stochastic,
    AUTHOR = {Baldi, Paolo},
     TITLE = {Stochastic Calculus},
    SERIES = {Universitext},
      NOTE = {An Introduction Through Theory and Exercises},
 PUBLISHER = {Springer, Cham},
      YEAR = {2017},
     PAGES = {xiv+627},
      ISBN = {978-3-319-62225-5},
   MRCLASS = {60-01 (60G42 60G44 60H05 60H10 60H30)},
  MRNUMBER = {3726894},
MRREVIEWER = {Josep\ Vives},
       DOI = {10.1007/978-3-319-62226-2},
URL = {https://doi.org/10.1007/978-3-319-62226-2},
}

@article {flandoli1994dissipativity,
    AUTHOR = {Flandoli, Franco},
     TITLE = {Dissipativity and invariant measures for stochastic
              {N}avier-{S}tokes equations},
   JOURNAL = {NoDEA Nonlinear Differential Equations Appl.},
  FJOURNAL = {NoDEA. Nonlinear Differential Equations and Applications},
    VOLUME = {1},
      YEAR = {1994},
    NUMBER = {4},
     PAGES = {403--423},
      ISSN = {1021-9722,1420-9004},
   MRCLASS = {35R60 (35Q30 60H15)},
  MRNUMBER = {1300150},
MRREVIEWER = {Jan\ I.\ Seidler},
       DOI = {10.1007/BF01194988},
URL = {https://doi.org/10.1007/BF01194988},
}

@article {giga1981analyticity,
    AUTHOR = {Giga, Yoshikazu},
     TITLE = {Analyticity of the Semigroup Generated by the {S}tokes
              Operator in {$L\sb{r}$} Spaces},
   JOURNAL = {Math. Z.},
  FJOURNAL = {Mathematische Zeitschrift},
    VOLUME = {178},
      YEAR = {1981},
    NUMBER = {3},
     PAGES = {297--329},
      ISSN = {0025-5874,1432-1823},
   MRCLASS = {47D05 (35Q20)},
  MRNUMBER = {635201},
MRREVIEWER = {A.\ Pazy},
       DOI = {10.1007/BF01214869},
URL = {https://doi.org/10.1007/BF01214869},
}

@book {taylor1996partial,
    AUTHOR = {Taylor, Michael E.},
     TITLE = {Partial Differential Equations. {I}},
    SERIES = {Applied Mathematical Sciences},
    VOLUME = {115},
      NOTE = {Basic Theory},
 PUBLISHER = {Springer-Verlag, New York},
      YEAR = {1996},
     PAGES = {xxiv+563},
      ISBN = {0-387-94653-5},
   MRCLASS = {35-01 (46N20 47N20 58Gxx)},
  MRNUMBER = {1395148},
MRREVIEWER = {Luigi\ Rodino},
       DOI = {10.1007/978-1-4684-9320-7},
URL = {https://doi.org/10.1007/978-1-4684-9320-7},
}

@article {hairer2011atheory,
    AUTHOR = {Hairer, Martin and Mattingly, Jonathan C.},
     TITLE = {A Theory of Hypoellipticity and Unique Ergodicity for
              Semilinear Stochastic {PDE}s},
   JOURNAL = {Electron. J. Probab.},
  FJOURNAL = {Electronic Journal of Probability},
    VOLUME = {16},
      YEAR = {2011},
     PAGES = {no. 23, 658--738},
      ISSN = {1083-6489},
   MRCLASS = {60H15 (35H10 35R60 60H07)},
  MRNUMBER = {2786645},
MRREVIEWER = {Feng-Yu\ Wang},
       DOI = {10.1214/EJP.v16-875},
URL = {https://doi.org/10.1214/EJP.v16-875},
}

@article {Fujiwaramorimoto1977Helmholtz,
    AUTHOR = {Fujiwara, Daisuke and Morimoto, Hiroko},
     TITLE = {An {$L\sb{r}$}-theorem of the {H}elmholtz decomposition of
              vector fields},
   JOURNAL = {J. Fac. Sci. Univ. Tokyo Sect. IA Math.},
  FJOURNAL = {Journal of the Faculty of Science. University of Tokyo.
              Section IA. Mathematics},
    VOLUME = {24},
      YEAR = {1977},
    NUMBER = {3},
     PAGES = {685--700},
      ISSN = {0040-8980},
   MRCLASS = {35Q99},
  MRNUMBER = {492980},
MRREVIEWER = {Marjorie\ McCracken},
}

@article {brzezniak2022ergodicity,
    AUTHOR = {Brze\'{z}niak, Zdzis{\l}aw and Komorowski, Tomasz and Peszat,
              Szymon},
     TITLE = {Ergodicity for stochastic equations of {N}avier-{S}tokes type},
   JOURNAL = {Electron. Commun. Probab.},
  FJOURNAL = {Electronic Communications in Probability},
    VOLUME = {27},
      YEAR = {2022},
     PAGES = {Paper No. 4, 10},
      ISSN = {1083-589X},
   MRCLASS = {60H15 (35Q30)},
  MRNUMBER = {4368698},
MRREVIEWER = {Qi\ L\"{u}},
       DOI = {10.1214/21-ecp443},
URL = {https://doi.org/10.1214/21-ecp443},
}

@article {goldys2005exponential,
    AUTHOR = {Goldys, B. and Maslowski, B.},
     TITLE = {Exponential ergodicity for stochastic {B}urgers and 2{D}
              {N}avier-{S}tokes equations},
   JOURNAL = {J. Funct. Anal.},
  FJOURNAL = {Journal of Functional Analysis},
    VOLUME = {226},
      YEAR = {2005},
    NUMBER = {1},
     PAGES = {230--255},
      ISSN = {0022-1236,1096-0783},
   MRCLASS = {35R60 (35Q30 35Q53 37L55 60H10 76D05 76M35)},
  MRNUMBER = {2158741},
MRREVIEWER = {Ana\ Bela\ Cruzeiro},
       DOI = {10.1016/j.jfa.2004.12.009},
URL = {https://doi.org/10.1016/j.jfa.2004.12.009},
}

@article {dong2011ergodicity,
    AUTHOR = {Dong, Zhao and Xie, Yingchao},
     TITLE = {Ergodicity of stochastic 2{D} {N}avier-{S}tokes equation with
              {L}\'{e}vy noise},
   JOURNAL = {J. Differential Equations},
  FJOURNAL = {Journal of Differential Equations},
    VOLUME = {251},
      YEAR = {2011},
    NUMBER = {1},
     PAGES = {196--222},
      ISSN = {0022-0396,1090-2732},
   MRCLASS = {37L55 (35Q30 35R60 60G51 60H15 76D06 76M35)},
  MRNUMBER = {2793269},
MRREVIEWER = {Kumarasamy\ Sakthivel},
       DOI = {10.1016/j.jde.2011.03.015},
URL = {https://doi.org/10.1016/j.jde.2011.03.015},
}

@article {fujitamorimoto1970fractional,
    AUTHOR = {Fujita, Hiroshi and Morimoto, Hiroko},
     TITLE = {On Fractional Powers of the Stokes Operator},
   JOURNAL = {Proc. Japan Acad.},
  FJOURNAL = {Proceedings of the Japan Academy},
    VOLUME = {46},
      YEAR = {1970},
     PAGES = {1141--1143},
      ISSN = {0021-4280},
   MRCLASS = {47F05 (35Q99)},
  MRNUMBER = {296755},
MRREVIEWER = {D.\ E.\ Edmunds},
       URL = {http://projecteuclid.org/euclid.pja/1195526510},
}

@book {Lunardi95Analyticsemigroups,
    AUTHOR = {Lunardi, Alessandra},
     TITLE = {Analytic Semigroups and Optimal Regularity in Parabolic Problems},
    SERIES = {Modern Birkh\"{a}user Classics},
      NOTE = {[2013 reprint of the 1995 original] [MR1329547]},
 PUBLISHER = {Birkh\"{a}user/Springer Basel AG, Basel},
      YEAR = {1995},
     PAGES = {xviii+424},
      ISBN = {978-3-0348-0556-8},
   MRCLASS = {47D06 (01A75 34G20 35Kxx 46M35 46N20 47N20 58D25)},
  MRNUMBER = {3012216},
}

\end{document}